\documentclass[12pt,a4paper,reqno]{amsart}
\usepackage[utf8]{inputenc}
\usepackage{amsthm}
\usepackage{amssymb}
\usepackage[abbrev]{amsrefs}
\usepackage[dvipsnames]{xcolor}
\usepackage{bm}
\usepackage[shortlabels]{enumitem}
\usepackage{hyperref}
\usepackage{bbm}
\newtheorem{thm}{}[section]
\newtheorem{theorem}[thm]{Theorem}
\newtheorem{corollary}[thm]{Corollary}
\newtheorem{lemma}[thm]{Lemma}
\newtheorem{proposition}[thm]{Proposition}

\theoremstyle{definition}

\theoremstyle{remark}
\newtheorem{remark}[thm]{Remark}

\numberwithin{equation}{section}
\allowdisplaybreaks
\newcommand{\Lip}{\ensuremath{\mathrm{Lip}}}
\newcommand{\RTO}{\ensuremath{\mathcal{R}}}
\newcommand{\TO}{\ensuremath{\mathcal{T}}}
\newcommand{\tqg}{\ensuremath{\bm{K_{tq}}}}
\newcommand{\fqg}{\ensuremath{\bm{K_{fq}}}}
\newcommand{\qglc}{\ensuremath{\bm{K_{ql}}}}
\newcommand{\ucc}{\ensuremath{\bm{K_{uc}}}}
\newcommand{\lucc}{\ensuremath{\bm{K_{lu}}}}
\newcommand{\lpu}{\ensuremath{\bm{K_{lp}}}}
\newcommand{\sch}{\ensuremath{\bm{K_{b}}}}
\newcommand{\FF}{\ensuremath{\mathbb{F}}}
\newcommand{\RR}{\ensuremath{\mathbb{R}}}
\newcommand{\CC}{\ensuremath{\mathbb{C}}}
\newcommand{\C}{\ensuremath{\mathbf{C}}}
\newcommand{\udf}{\ensuremath{\bm{\varphi_{u}}}}

\newcommand{\NN}{\ensuremath{\mathbb{N}}}
\newcommand{\Nt}{\ensuremath{\mathcal{N}}}
\newcommand{\xx}{\ensuremath{\bm{x}}}

\newcommand{\yy}{\ensuremath{\bm{y}}}

\newcommand{\ww}{\ensuremath{\bm{w}}}
\newcommand{\zz}{\ensuremath{\bm{z}}}
\newcommand{\XX}{\ensuremath{\mathbb{X}}}
\newcommand{\XB}{\ensuremath{\mathcal{X}}}
\newcommand{\YB}{\ensuremath{\mathcal{Y}}}

\newcommand{\ZB}{\ensuremath{\mathcal{Z}}}

\newcommand{\YY}{\ensuremath{\mathbb{Y}}}

\newcommand{\UU}{\ensuremath{\mathbb{U}}}
\newcommand{\Ind}{\ensuremath{\mathbbm{1}}}

\newcommand{\EE}{\ensuremath{\mathbb{E}}}
\newcommand{\Cu}{\ensuremath{\mathcal{Q}}}

\DeclareMathOperator{\osc}{osc}
\DeclareMathOperator{\sgn}{sign}

\DeclareMathOperator{\supp}{supp}
\newcommand{\enbrace}[1]{\left\lbrace#1\right\rbrace}
\newcommand{\enpar}[1]{\left(#1\right)}
\newcommand{\abs}[1]{\left\lvert#1\right\rvert}
\newcommand{\norm}[1]{\left\lVert#1\right\rVert}

\newcommand{\ceil}[1]{\left\lceil #1 \right\rceil}
\AtBeginDocument{\def\MR#1{}}
\newcommand{\Miguel}[1]{\textcolor[RGB]{252,10,12}{ #1}}

\begin{document}
\title[Linear vs. nonlinear partial unconditionality]{Linear versus nonlinear forms\\ of partial unconditionality of bases}
\thanks{F. Albiac acknowledges the support of the Spanish Ministry for Science and Innovation under Grant PID2019-107701GB-I00 for \emph{Operators, lattices, and structure of Banach spaces}. M. Berasategui was supported by ANPCyT PICT-2018-04104.}
\author[Albiac]{Fernando Albiac}
\address{Department of Mathematics, Statistics, and Computer Sciencies--InaMat2\\
Universidad P\'ublica de Navarra\\
Campus de Arrosad\'{i}a\\
Pamplona\\
31006 Spain}
\email{fernando.albiac@unavarra.es}

\author[Ansorena]{Jos\'e L. Ansorena}
\address{Department of Mathematics and Computer Sciences\\
Universidad de La Rioja\\
Logro\~no\\
26004 Spain}
\email{joseluis.ansorena@unirioja.es}

\author[Berasategui]{Miguel Berasategui}
\address{Miguel Berasategui\\
IMAS - UBA - CONICET - Pab I, Facultad de Ciencias Exactas y Naturales\\
Universidad de Buenos Aires\\
(1428), Buenos Aires, Argentina}
\email{mberasategui@dm.uba.ar}

\begin{abstract}
The main results in this paper contribute to bringing to the fore novel underlying connections between the contemporary concepts and methods springing from greedy approximation theory with the well-established techniques of classical Banach spaces. We do that by showing that bounded-oscillation unconditional bases, introduced by Dilworth et al. in 2009 in the setting of their search for extraction principles of subsequences verifying partial forms of unconditionality, are the same as truncation quasi-greedy bases, a new breed of bases that appear naturally in the study of the performance of the thresholding greedy algorithm in Banach spaces. We use this identification to provide examples of bases that exhibit that bounded unconditionality is a stronger condition than Elton's near unconditionality. We also take advantage of our arguments to provide examples that allow us to tell apart certain types of bases that verify either debilitated unconditionality conditions or weaker forms of quasi-greediness in the context of abstract approximation theory.
\end{abstract}
\maketitle
\section{Introduction}\noindent
Knowing the structure of bases and basic sequences in a given Banach space is of outmost importance in the understanding of its geometry. Within the study of basis related properties, the possibility to extract certain subsequences with additional features (such as being unconditional) is a technique of major usage throughout classical Banach space theory. One of the most celebrated results in this direction was Rosenthal's $\ell_1$-theorem, which states that every bounded sequence in a Banach space $\XX$ with no copies of $\ell_1$ has a weakly Cauchy subsequence (see \cite{Rosenthal1974}). Rosenthal's dichotomy inspired further work on the subject that led very naturally to conjecture whether any weakly null sequence in a Banach spaces possesses an unconditional subsequence. In view of Maurey and Rosenthal discovery of a counterexample \cite{MaureyRosenthal1977}, the general extraction problem turned into one where the subsequence was required to fulfil weakened forms of unconditionality. In this milieu, Elton \cite{Elton1978} proved that every normalized weakly null sequence of a Banach space has a subsequence that is ``close'' to being unconditional in a sense that we will specify below, and which led to the introduction of the concept of nearly unconditional bases. Forty years later, Dilworth et al.\ \cite{DOSZ2009} improved Elton's result by showing that the subsequence we extract can be chosen to satisfy a more demanding unconditionality property called bounded-oscillation unconditionality. To state and contextualize these results, we shall introduce some preliminary terminology.

Let $\XX$ be a Banach (or more generally a quasi-Banach) space over the real or complex field $\FF$. In this paper we will use the term \emph{basis} to refer to a norm-bounded sequence $\XB=(\xx_n)_{n=1}^\infty$ in $\XX$ whose linear span is dense in $\XX$ and for which there is a (unique) norm-bounded sequence of linear functionals $\XB^*=(\xx_n^*)_{n=1}^\infty$ in $\XX^*$ biorthogonal to $\XB$, that is $\xx_n^*(\xx_k)=\delta_{n,k}$ for all positive integers $n$ and $k$. The symbol $A\subset_f B$ will mean that $A$ is a finite subset of $B$. Given $A\subset_f \NN$, the \emph{projection on $A$ relative to the basis $\XB$} is the bounded linear map $S_A\colon\XX\to\XX$ given by
\[
S_{A}(f)=\sum_{n\in A}\xx_n^*(f)\xx_n, \quad f\in\XX.
\]
Set $\EE=\{\lambda\in\FF \colon \abs{\lambda}=1\}$. The \emph{sign} of $f\in\XX$ with respect to $\XB$ is the sequence $\varepsilon(f)\in\EE^\NN$ given by
\[
\varepsilon(f)=(\sgn(\xx_n^*(f)))_{n=1}^\infty,
\]
where $\sgn(\lambda)=\lambda/\abs{\lambda}$ if $\lambda\in\FF\setminus\{0\}$ and $\sgn(0)=1$. The \emph{support} of $f\in\XX$ is the set
\[
\supp(f)=\{n\in\NN \colon \xx_n^*(f)\not=0\},
\]
and we define the $\ell_\infty$-norm of $f\in\XX$ with respect to $\XB$ as
\[
\norm{f}_\infty=\sup_{n\in\NN} \abs{\xx_n^*(f)}.
\]

Note that, since we are assuming that $\XB^*$ is bounded, $\lim_n \xx_n^*(f)=0$ for all $f\in\XX$. Therefore, $\norm{f}_\infty<\infty$ and
\[
\abs{\{n\in\NN \colon \abs{\xx_n^*(f)}\ge d\}}<\infty
\]
for every $f\in\XX$ and $d>0$.

The \emph{oscillation} of $f\in\XX$ on a nonempty set $A\subset_f\NN$ is the number
\[
\osc(f,A)=\sup_{n\in A} \frac{\max_{n\in A} \abs{\xx_n^*(f)} }{\min_{n\in A} \abs{\xx_n^*(f)}},
\]
with the convention that $0/0=1$ and $a/0=\infty$ if $a>0$.

Now, given numbers $0<a\le b <\infty$ and $f\in \XX$ we put
\[
A(f,a,b):=\{n\in\NN \colon a\le \abs{\xx_n^*(f)}\le b\},
\]
so that $\osc(f,A)\le b/a$ whenever $A\subset A(f,a,b)$. The basis $\XB$ is said to be \emph{nearly unconditional} if for each $t\in(0,1]$ there is a constant $C\in(0,\infty)$ such that $\norm{S_A(f)}\le C\norm{f}$
for all $f\in\XX$, all $s\in(0,\infty)$ with $\norm{f}_\infty\le s$, and all $A\subset A(f,ts,s)$. Given $0<t\le 1$, we define $\phi(t)$ as the smallest value of $C$.

In turn, the basis is said to be \emph{bounded-oscillation unconditional} if for every $1\le d\le D<\infty$ there is a constant $C$ such that
$\norm{S_A(f)} \le C \norm{f}$ whenever $A$ is a nonempty subset of $\NN$ such that
\begin{itemize}[leftmargin=*]
\item $\osc(f,A)\le D$, and
\item there exists $n\in\NN$ and a partition $(A_j)_{j=1}^n$ of $A$ into nonempty subsets with
\begin{itemize}[label=$ \rhd$]
\item $\osc(f,A_j)\le d$ for all $1\le j\le n$,
\item $n\le \min(A_1)$, and
\item $\max(A_j)<\min(A_{j+1})$ for $1\le j\le n-1$.
\end{itemize}
\end{itemize}
If these conditions hold for a pair $(D,d)$ we say that the basis is $(D,d)$-bounded-oscillation unconditional, and we denote by $\beta(D,d)$ the smallest value of the constant $C$.

Since a basis $\XB$ is unconditional if and only if $\sup_{\abs{A}<\infty} \norm{S_A}<\infty$, unconditional bases are in particular bounded-oscillation unconditional and nearly unconditional. Observe that given $1\le d<\infty$, a basis is $(d,d)$-bounded-oscillation unconditional if and only if there is a constant $C$ such that $\norm{S_A(f)}\le C\norm{f}$ for all $f\in\XX$, all $s\in(0,\infty)$, and all $A\subset A(f,s/d,s)$. Hence, bounded-oscillation unconditional bases are nearly unconditional. Quantitatively,
\[
\phi(t) \le \Phi(t):=\beta(1/t,1/t), \quad 0<t\le 1.
\]
Clearly, both functions $\Phi$ and $\phi$ are non-increasing.

In this language we can now state the extraction principles we advertised before.\begin{theorem}[\cites{Elton1978,Odell1980}, cf. \cite{DKK2003}*{Theorem 5.1 and Remark 5.2}]\label{thm:EPNU}
Every normalized weakly null sequence in a real Banach space has a nearly unconditional subsequence $\XB$. Moreover, there is a universal constant $C$ such that
the basic sequence $\XB$ satisfies
\[
\phi(t) \le C (1-\log t), \quad 0<t\le 1.
\]
\end{theorem}

\begin{theorem}[ \cite{DOSZ2009}*{Theorem 2.1}]\label{thm:EPBOU}
Given $1\le d\le D<\infty$ and $C>8$, every normalized weakly null sequence in a real Banach space has a $(D,d)$-bounded-oscillation unconditional subsequence with $\beta(D,d)\le Cd$.
\end{theorem}

Dilworth et al. \cite{DKK2003} realized the connection between near unconditionality and the performance of the thresholding greedy algorithm. Delving deeper into the matter, near unconditionality has been recently characterized as a threshold-free greedy-like property.

\begin{theorem}[\cite{AAB2023}]\label{thm:AAB2023}
A basis is nearly unconditional if and only if it is quasi-greedy for largest coefficients.
\end{theorem}

Let us introduce the necessary terminology to properly understand this result. A subset $A\subset_f\NN$ is said to be a \emph{greedy set} of $f\in\XX$ with respect to a basis $\XB$ if $\abs{\xx_n^*(f)} \ge \abs{\xx_k^*(f)}$ for all $n\in A$ and all $k\in\NN\setminus A$. The basis $\XB$ is said to be \emph{quasi-greedy} if there is a constant $C\ge 1$ such that $\norm{S_A(f)} \le C \norm{f}$ for all $f\in\XX$ and all greedy sets $A$ of $f$. If this inequality holds under the extra assumption that $\osc(f,A)=1$, we say that $\XB$ is \emph{quasi-greedy for largest coefficients}, and we denote by $\qglc$ the optimal constant $C$. With the notation
\[
\Ind_{\varepsilon,A}=\Ind_{\varepsilon,A}[\XB,\XX]=\sum_{n\in A} \varepsilon_n \, \xx_n,
\]
we have that $\XB$ is quasi-greedy for largest coefficients with constant $C$ if and only if
\[
\norm{ \Ind_{\varepsilon,A}} \le C\norm{ \Ind_{\varepsilon,A}+ f}
\]
for all $A\subset_f\NN$, all $\varepsilon\in\EE^A$, and all $f\in\XX$ with $\norm{f}_\infty\le 1$ and $A\cap \supp(f)=\emptyset$. Bases $\XB$ that satisfy the more restrictive condition
\[
\min_{n\in A} \abs{\xx_n^*(f)} \norm{\Ind_{\varepsilon(f),A}} \le C \norm{f}
\]
for all $f\in\XX$ and all greedy sets $A$ of $f$ are called \emph{truncation quasi-greedy} with constant $C$. Although this term was recently coined in \cite{AABBL2022}, it captures a property of quasi-greedy bases of Banach spaces that was already implicit in the early days of the theory (see \cite{DKKT2003}*{Lemma 2.2} and \cite{Woj2000}*{Proof of Theorem 3}). The fact that quasi-greedy bases are truncation quasi-greedy also in the lack of local convexity is far from trivial and was proved in \cite{AABW2021}.
\begin{theorem}[\cite{AABW2021}*{Theorem 4.13}]\label{thm:QGvsTQG}
Assume that $\XB$ is a quasi-greedy basis of a quasi-Banach space $\XX$. Then $\XB$ is truncation quasi-greedy.
\end{theorem}
The first example that illustrates that truncation quasi-greedy bases need not be quasi-greedy can be found in
\cite{DKK2003}*{Example 4.8}, where the authors constructed a basis that dominates the unit vector system of weak-$\ell_1$ (and so it is truncation quasi-greedy) but it is not quasi-greedy. In this regard, it must be mentioned that it was recently proved that bidemocratic bases need not be quasi-greedy \cite{AABBL2023}*{Theorem 3.6}. Since bidemocratic bases are truncation quasi-greedy, this result yields in particular the existence of truncation quasi-greedy bases that are not quasi-greedy. Despite this dissimilitude, in most situations the only property of quasi-greedy bases that one needs in practice is that they are truncation quasi-greedy. To illustrate this statement, see, e.g., \cite{AAW2021b}*{Theorem 5.1}, \cite{AABW2021}*{Proposition 10.17}, \cite{AAW2021}*{Theorems 3.7 and 4.3}, \cite{AABBL2022}*{Corollary 4.5} or \cite{AABBL2023}*{Theorem 2.5}.

In this paper we show that bounded-oscillation unconditionality is also a greedy-like property. To be specific, we prove the following result.

\begin{theorem}\label{thm:main}
Given a basis $\XB$ of a quasi-Banach space, the following are equivalent.
\begin{enumerate}[label=(\roman*), leftmargin=*, widest=iii]
\item $\XB$ is bounded-oscillation unconditional.
\item $\XB$ is $(1,1)$-bounded-oscillation unconditional.
\item $\XB$ is truncation quasi-greedy.
\end{enumerate}
\end{theorem}

If we concede that truncation quasi-greedy bases are close to being quasi-greedy, Theorem~\ref{thm:main} sheds light onto the long-standing conjecture that any weakly null normalized sequence has a quasi-greedy subsequence (see \cite{DKK2003}*{Theorem 5.4} and \cite{DOSZ2009}*{Problem 5}). Also, Theorem~\ref{thm:main} upgrades the relevance of determining whether every quasi-greedy for largest coefficients basis is truncation quasi-greedy (see \cite{AAB2023}*{Question 5.1}). We solve the latter question in the negative.

\begin{theorem}\label{thm:example}
There is a Schauder basis $\XB$ of a Banach space $\XX$ that is nearly unconditional but is not bounded-oscillation unconditional. Moreover, given $1\le p<\infty$, $\XB$ and $\XX$ can be chosen so that
\begin{equation}\label{eq:FFlp}
\norm{\Ind_{\varepsilon,A}}\approx \abs{A}^{1/p}, \quad A\subset_f\NN, \, \varepsilon\in\EE^A.
\end{equation}
\end{theorem}

Notice that Theorem~\ref{thm:example} evinces that Theorem~\ref{thm:EPBOU} constitutes a real improvement with respect to Theorem~\ref{thm:EPNU}. Note also that \eqref{eq:FFlp} is a democracy-like property. For the record, a basis $\XB$ is said to be \emph{super-democratic} if there is a constant $C$ so that $\norm{\Ind_{\varepsilon,A}}\le C \norm{\Ind_{\delta,B}}$ for all finite subsets $A$ and $B$ of $\NN$ with $\abs{A}\le \abs{B}$, all $\varepsilon\in\EE^A$, and all $\delta\in\EE^B$. If this inequality holds in the particular case that $\varepsilon$ and $\delta$ are constant, we say that $\XB$ is \emph{democratic}. The \emph{fundamental function} of $\XB$ is the function defined for each $m\in\NN$ as
\[
\udf(m)=\sup\enbrace{ \norm{\Ind_{\varepsilon,A}} \colon \abs{A}\le m, \, \varepsilon\in\EE^A }.
\]
With this terminology, \eqref{eq:FFlp} says that $\XB$ is superdemocratic and its fundamental function grows as $(m^{1/p}) _{m=1}^\infty$.

We close this introductory section by briefly explaining the structure of the paper. Section~\ref{sect:BouvsTQG} will be devoted to proving Theorem~\ref{thm:main} and establishing some properties of the function $\Phi$ associated with bounded-oscillation unconditionality. In Section~\ref{sect:EltonNotDilworth}, we prove Theorem~\ref{thm:example}. In Section~\ref{TOvsQGLC}, we take advantage of the techniques we develop to tell apart other greedy-like properties that have appeared in the literature.

\section{Characterization of bounded-oscillation unconditional bases}\label{sect:BouvsTQG}\noindent
We start our study by noticing that if a basis $\XB$ is $(1/t,1/t)$-bounded-oscillation unconditional for some $0<t \le 1$, then it is $(1/t,d)$-bounded-oscillation unconditional for every $1\le d \le 1/t$, and
\[
\beta(1/t,d)\le \Phi(t), \quad 1\le d \le \frac{1}{t}<\infty.
\]
Hence, the basis $\XB$ is bounded-oscillation unconditional if and only if it is $(1/t,1/t)$-bounded-oscillation unconditional for all $0<t \le 1$.

\begin{lemma}\label{lem:Philogconvex}
Let $\XB$ be a basis of a $p$-Banach space $\XX$, $0<p\le 1$. The function $s\mapsto \Phi^p(e^{-s})$, $s>0$, is subadditive. That is, if $\XB$ is $(1/t_j,1/t_j)$-bounded-oscillation unconditional, $0<t_j<1$, $j=1$, $2$, then it is $(1/(t_1t_2),(1/(t_1t_2))$-bounded-oscillation unconditional, and
\[
\Phi(t_1t_2) \le \enpar{\Phi^p(t_1)+\Phi^p(t_2)}^{1/p}.
\]
\end{lemma}

\begin{proof}
Let $f\in\XX$ and $A\subset_f \NN$ with $\osc(f,A) \le 1/(t_1 t_2)$. There is a partition $(A_1,A_2)$ of $A$ with $\osc(f,A_j) \le 1/t_j$, $j=1$, $2$. Indeed, we can choose,
\begin{align*}
A_1&=\enbrace{n\in A \colon b \le \abs{\xx_n^*(f)} \le b t_1}, \\
A_2&=\enbrace{n\in A \colon b t_1 <\abs{\xx_n^*(f)} \le b t_1 t_2},
\end{align*}
where $b:=\min_{n\in A}\abs{\xx_n^*(f)}$. Since
\[
\norm{S_A(f)} \le
\enpar{\norm{S_{A_1}(f)}^p + \norm{S_{A_2}(f)}^p }^{1/p}
\le \enpar{\Phi^p(t_1)+\Phi^p(t_2)}^{1/p},
\]
we are done.
\end{proof}

\begin{lemma}\label{lem:SomeVsAll}
Let $\XB$ be a basis of a $p$-Banach space $\XX$, $0<p\le 1$. Suppose that $\XB$ is $(1/s,1/s)$-bounded-oscillation unconditional for some $0<s<1$. Then $\XB$ is
bounded-oscillation unconditional. Quantitatively,
\[
\Phi(t) \le \ceil{ \log_s(t)}^{1/p} \Phi(s), \quad 0<t<1.
\]
\end{lemma}

\begin{proof}
By Lemma~\ref{lem:Philogconvex}, $\Phi(s^n)\le n^{1/p}\Phi(s)$ for all $n\in\NN$. Since $\Phi$ is non-increasing, we are done.
\end{proof}

We are ready to tackle the proof of Theorem~\ref{thm:main}. We first recall the following auxiliary result from \cite{AABW2021}, which plays an important role in the study of the thresholding greedy algorithm in general quasi-Banach spaces. \begin{lemma}[\cite{AABW2021}*{Theorem 2.2}]\label{lem:pconvex}
Given a finite family $(f_n)_{n\in \Nt}$ in a $p$-Banach space, $0<p\le 1$, we have
\[
\norm{\sum_{n\in\Nt} a_n \, f_n} \le A_p \sup_{A\subset\Nt} \norm{\sum_{n\in A} f_n} ,\quad \abs{a_n}\le 1,
\]
where
\[
A_p = \frac{1}{(2^p -1)^{1/p}}.
\]
\end{lemma}

\begin{proof}[Proof of Theorem~\ref{thm:main}]
Assume without loss of generality that $\XX$ is a $p$-Banach space for some $0<p\le 1$. Let $\XB=(\xx_n)_{n=1}^\infty$ be a basis of $\XX$ with dual basis $\XB^*=(\xx_n^*)_{n=1}^\infty$.

Suppose that $\XB$ is truncation quasi-greedy with constant $C$. Pick $A\subset_f\NN$, $\varepsilon\in\EE^A$ and $f\in\XX$ with $A\cap \supp(f)=\emptyset$. Set $B=\{n\in\NN \colon \abs{\xx_n^*(f)} > 1\}$. Since both $B$ and $A\cup B$ are greedy sets of $\Ind_{\varepsilon,A}+f$,
\[
\norm{\Ind_{\varepsilon,A}}^p\le \norm{\Ind_{\varepsilon,A\cup B}}^p+\norm{\Ind_{\varepsilon,B}}^p
\le 2 C^p \norm{\Ind_{\varepsilon,A}+f}^p.
\]
Hence, $\XB$ is $(1,1)$-bounded-oscillation unconditional.

Suppose now that $\XB$ is $(1,1)$-bounded-oscillation unconditional and set $C=\Phi(1)$. Let $0<t < 1$ and $f\in\XX$ with $\osc(f,A)\le 1/t$. By rescaling, we can assume that $\norm{S_A(f)}_\infty=1$. Choose $B\subset A$ such that
\[
\norm{\Ind_{\varepsilon(f),E} } \le \norm{\Ind_{\varepsilon(f),B} }, \quad E\subset A.
\]
By Lemma~\ref{lem:pconvex},
\begin{equation}\label{eq:aux}
\norm{S_A(f)} \le A_p \norm{\Ind_{\varepsilon(f),B} }
\end{equation}
and
\[
\norm{ \Ind_{\varepsilon(f),B} -S_B(f)}\le A_p(1-t) \norm{\Ind_{\varepsilon(f),B}}.
\]
Using the last inequality we obtain
\begin{align*}
\norm{\Ind_{\varepsilon(f),B} }^p &\le C^p \norm{ \Ind_{\varepsilon(f),B} + f-S_B(f) }^p\\
&\le C^p \norm{f}^p + C^p \norm{ \Ind_{\varepsilon(f),B}-S_B(f)}^p\\
&\le C^p \norm{f}^p + A_p^pC^p (1-t)^p \norm{\Ind_{\varepsilon(f),B} }^p.
\end{align*}
If $t>1-1/(CA_p)$ this inequality gives
\[
\norm{\Ind_{\varepsilon(f),B} } \le C (1- A_p^p C^p (1-t)^p)^{-1/p} \norm{f}.
\]
Combining with \eqref{eq:aux} gives that $\XB$ is $(1/t,1/t)$-bounded-oscillation unconditional. By Lemma~\ref{lem:SomeVsAll}, $\XB$ is bounded-oscillation unconditional.

Suppose now that $\XB$ is bounded-oscillation unconditional. Pick $0<t<1$ arbitrary. Let $A$ be a greedy set of $f\in\XX$. Set $b=\min_{n\in A} \abs{\xx_n^*(f)}$ and
assume that $b>0$; otherwise there is nothing to prove. Put
\[
A_k=\enbrace{n\in A \colon b t^{1-k} \le \abs{\xx_n^*(f)} < b t^{-k}}, \quad k\in\NN,
\]
so that $(A_k)_{k=1}^\infty$ is a partition of $A$. For each $k\in\NN$, choose $B_k\subset A_k$ such that
\[
\norm{S_B(f)}\le \norm{S_{B_k}(f)}, \quad B\subset A_k.
\]
By Lemma~\ref{lem:pconvex},
\[
\norm{\sum_{n\in A_k} a_n\, \xx_n^*(f) \, \xx_n} \le A_p\norm{S_{B_k}(f)}, \quad 0 \le a_n \le 1.
\]
Since $\osc(B_k,f)\le 1/t$ for all $k\in\NN$,
\begin{align*}
b^p \norm{\Ind_{\varepsilon(f), A}(f)}^p&\le \sum_{k=1}^\infty b^p\norm{\Ind_{\varepsilon(f), A_k}(f)}^p\\
&\le \sum_{k=1}^\infty t^{(k-1)p} A_p^p \norm{S_{B_k}(f)}^p\\
&\le \Phi^p(t) \norm{f}^p A_p^p \sum_{k=1}^\infty t^{(k-1)p}\\
&= \frac{A_p^p}{1-t^p} \Phi^p(t) \norm{f}^p.
\end{align*}
Hence, $\XB$ is truncation quasi-greedy.
\end{proof}

\begin{remark} \label{remark: QG->TQG}
We can use Theorem~\ref{thm:main} to give an alternative proof of Theorem~\ref{thm:QGvsTQG}. In fact, it is known that quasi-greedy bases are $(1,1)$-bounded-oscillation unconditional (see \cite{DKKT2003}*{Proof of Theorem 5.4}). For the sake of completeness we include a new proof of this. Let $\XB$ a quasi-greedy basis with constant $C$ of a $p$-Banach space $\XX$, $0<p\le 1$. Let $A\subset_f\NN$, $\varepsilon\in\EE^A$ and $f\in\XX$ with $\supp(f)\cap A$. Set $B=\{n\in\NN \colon \abs{\xx_n^*(f)} > 1\}$. Since both $B$ and $A\cup B$ are greedy sets of $g:=\Ind_{\varepsilon,A}+f$,
\[
\norm{\Ind_{\varepsilon,A}}^p=\norm{S_A(g)}^p
\le \norm{S_{A\cup B}(g)}^p+\norm{S_B(g)}^p
\le 2 C^p \norm{g}^p.
\]
Therefore, $\XB$ is $(1,1)$-bounded-oscillation unconditional with $\Phi(1)\le 2^{1/p}C$.
\end{remark}

Before we continue, let us write down the quantitative estimates that we obtain in the proof of Theorem~\ref{thm:main}. Given a basis $\XB$ of a $p$-Banach space $\XX$, $0<p\le 1$, let $\tqg$ the optimal constant $C$ such that $\XB$ is truncation quasi-greedy with constant $C$. We have
\begin{align}
\Phi(1) & \le 2^{1/p} \tqg,\nonumber\\
\Phi(t) & \le A_p \inf_{1-1/(A_p\Phi(1))<s<1} \ceil{ \log_s(t)}^{1/p} \Phi(1) (1- A_p^p\Phi^p(1) (1-s)^p)^{-1/p}\label{eq:BOU},\\
\tqg &\le A_p\inf_{0<t<1} (1-t^p)^{-1/p} \Phi(t).\nonumber
\end{align}

The following ready consequence of \eqref{eq:BOU} improves \cite{AAB2023b}*{Theorem 6.5}, which establishes a similar estimate for the function $\phi$.

\begin{corollary}\label{corollaryPhilog}
Let $\XB$ be a bounded-oscillation unconditional basis of a $p$-Banach space $\XX$. There is a constant $C$ depending only on $\Phi(1)$ and $p$ such that
\[
\Phi(t)\le C\left(1-\log t\right)^{{1}/{p}}, \quad 0<t\le 1.
\]
\end{corollary}

\begin{remark}
Note that for every bounded-oscillation unconditional basis of a $p$-Banach space, either
\[
\Phi(t)\approx (1-\log t)^{1/p},\quad 0<t\le 1,
\]
or
\[
\limsup_{t\to 0}\frac{\Phi(t)}{(-\log t)^{1/p}}=0.
\]
Indeed, this follows by combining Corollary~\ref{corollaryPhilog} with the estimate
\[
\Phi(a)\le 2^{{1}/{p}}\Phi(t)\left(\frac{-\log a }{-\log t}\right)^{{1}/{p}}, \quad 0<a\le t<1,
\]
which follows at once from the fact that $\Phi\left(s^{n}\right)\le n^{1/p}\Phi\left(s\right)$ for all $n\in \NN$ and all $0<s<1$ (see the proof of Lemma~\ref{lem:SomeVsAll}).
\end{remark}

Next, we improve Theorem~\ref{thm:EPNU}. For the proof we observe that Theorem~\ref{thm:EPBOU} still holds for semi-normalized sequences.

\begin{corollary}\label{corollaryPhilog2}
From every semi-normalized weakly null sequence in a real Banach space we can extract a bounded-oscillation unconditional subsequence with
\[
\Phi(t) \le C (1-\log t), \quad 0<t\le 1,
\]
for some universal constant $C$.\end{corollary}

\begin{proof}
Note that Theorem~\ref{thm:EPBOU} still holds for semi-normalized sequences. Hence, the result follows by combining it with
Corollary~\ref{corollaryPhilog}.
\end{proof}

Next, we establish some smoothness properties of the function $\Phi$. They rely on an estimate that complements Lemma~\ref{lem:Philogconvex}.

\begin{lemma}
Let $\XB=(\xx_n)_{n=1}^\infty$ be a bounded-oscillation unconditonal basis of a $p$-Banach space $\XX$, $0<p\le 1$. For any $0<a\le b\le 1$ we have
\begin{equation}\label{eq:Berarategui}
\Phi(ab)\le \enpar{ (1-b)^p\Phi^{p}(b)+ \Phi^{p}(a) \enpar{ 1+ (1-b)^p \Phi^{p} (b) }}^{1/p}.
\end{equation}
\end{lemma}
\begin{proof}
Fix $f\in\XX$ and $A\subset \supp(f)$ with $\osc(f,A)\le (ab)^{-1}$. By rescaling we may assume $\norm{S_A(f)}{\infty}=1$. Let
\[
g:=f-(1-b)S_{D}(f),
\]
where $D=A(f,b,1)$. We have
\begin{equation}\label{eq:anso1}
\norm{ g}^p\le
\norm{f}^p+(1-b)^p \norm{S_{D}(f)}^p
\le ( 1+ (1-b)^p \Phi^{\, p} (b) ) \norm{f}^p.
\end{equation}
Now set $B:=A\cap D$. Since $ S_A(g)=S_A(f) -(1-b)S_B(f)$,
\begin{equation}\label{eq:anso2}
\norm{S_A(f)}^p \le \norm{S_A(g)}^p + (1-b)^p \norm{S_B(f)}^p.
\end{equation}

Since $\osc(g,A)\le a^{-1}$ and $\osc(f,B)\le b^{-1}$,
\begin{equation}\label{eq:anso3}
\norm{S_A(g)} \le \Phi(a) \norm{g} \quad \mbox{ and }\quad \norm{S_B(f)} \le \Phi(b) \norm{f}.
\end{equation}

Combining \eqref{eq:anso1}, \eqref{eq:anso2} and \eqref{eq:anso3} gives the desired inequality.
\end{proof}

\begin{corollary}\label{corollary pLipschitz}
Let $\XB$ be a bounded-oscillation unconditional basis of a $p$-Banach space $\XX$, $0<p\le 1$. Then

\begin{enumerate}[label=(\roman*), leftmargin=*, widest=iii]
\item \label{cont} $\Phi$ is continuous on the interval $(0,1]$.
\item \label{plip} $\Phi$ is $p$-Lipschitz on the interval $[c,d]$ for every $0<c<d<1$.
\item \label{lip} If $p=1$, for each $0<c<1$ $\Phi$ is Lipschitz on the interval $[c,1]$ with Lipschitz constant
\[
\Lip(\Phi,[c,1])\le \frac{\Phi(1) (1+\Phi(c))}{c}.
\]
\end{enumerate}
\end{corollary}
\begin{proof}
It is known that \eqref{eq:Berarategui} yields \ref{plip}, whereas \eqref{eq:Berarategui} and \ref{cont} yield \ref{lip} (see \cite{AAB2023}*{Proofs of Proposition 2.8 and Corollary 2.9}). Thus, we only need to prove that $\Phi$ is left-continuous at $1$. To that end, fix $0<t<1$, $f\in \XX\setminus\{0\}$ and $A\subset \supp(f)$ so that $\osc(f,A)\le 1/t$. By scaling, we may assume $\norm{S_A(f)}_{\infty}=1$. Pick $B\subset A$ so that
\begin{align*}
\norm{\Ind_{\varepsilon(f),D}}\le \norm{\Ind_{\varepsilon(f),B}}, \quad D\subset A.
\end{align*}
Since $\abs{ \xx_n^*(f)}\ge t$ for all $n\in A$, by Lemma~\ref{lem:pconvex},
\begin{align*}
\norm{S_A(f)}^p
&\le \norm{\Ind_{\varepsilon(f),A}}^p+\norm{\Ind_{\varepsilon(f),A}-S_A(f)}^p\\
&\le \norm{\Ind_{\varepsilon(f),B}}^p+A_p^p (1-t)^p\norm{\Ind_{\varepsilon(f),B}}^p \\
&=(1+A_p^p (1-t)^p) \norm{\Ind_{\varepsilon(f),B}}^p.
\end{align*}
On the other hand,
\begin{align*}
\norm{\Ind_{\varepsilon(f),B}}^p &\le \Phi^p(1)\norm{\Ind_{\varepsilon(f),B}-S_B(f)+f}^p\\
&\le \Phi^p(1) \left( \norm{\Ind_{\varepsilon(f),B}-S_B(f)}^{p} +\norm{f}^p\right)\\
&\le \Phi^p(1)\norm{f}^p+\Phi^p(1)A_p^p\Miguel{(1-t)}^p\norm{\Ind_{\varepsilon(f),B}}^p,
\end{align*}
whence
\[
\norm{\Ind_{\varepsilon(f),B}}\le \frac{\Phi(1)}{\left(1-\Phi^p(1)A_p^p(1-t)^p\right)^{{1}/{p}}}\norm{f}.
\]
Combining, we get $\norm{S_A(f)}\le C(t) \Phi(1) \norm{f}$, where
\[
C(t)= \frac{(1+A_p^p (1-t)^p)^{1/p}}{\left(1-\Phi^p(1)A_p^p\Miguel{(1-t)}^p\right)^{{1}/{p}}}.
\]
Taking the supremum over $A$ we conclude that $\Phi(t) \le C(t) \Phi(1)$. Since $\lim_{t\to 1} C(t)=1$, we are done.
\end{proof}

Our next result examines the behaviour of the function $\Phi$ when we rescale a truncation quasi-greedy basis. As we shall see, bounded-oscillation unconditionality is a very stable property under such perturbations, so that $\Phi$ only grows very slowly in a sense that will become clear below. Let us first introduce some notation we will use. Given a semi-normalized scalar sequence $\alpha=(\alpha_n)_{n=1}^\infty$ and a basis $\XB=(\xx_n)_{n=1}^\infty$ of a $p$-Banach space $\XX$, set
\[
\XB_\alpha=(\alpha_n\xx_n)_{n=1}^\infty\quad\text{and}\quad K_{\alpha}=\frac{\inf_{n\in \NN}\abs{\alpha_n}}{\sup_{n\in \NN} \abs{\alpha_n}}.
\]
To make clear that $\Phi$ is the bounded-oscillation unconditionality function of $\XB$, we write $\Phi=\Phi[\XB,\XX]$. We will use the same convention with other parameters.

\begin{proposition}\label{lemma seminormalized1}
Let $\XB$ be a basis of a $p$-Banach space $\XX$, and $\alpha$ be a semi-normalized scalar sequence. If $\XB$ is bounded-oscillation unconditional, so is $\XB_{\alpha}$, with
\[
\Phi[\XB_\alpha,\XX](t) \le \left(\Phi^p[\XB,\XX](t+\Phi^p[\XB,\XX] (K_\alpha)\right)^{1/p}
\]
for all $0<t\le 1$.
\end{proposition}
\begin{proof}
Pick $0<t \le 1$, $f\in \XX$ and $A\subset_f \NN$ with $\osc[\XB_{\alpha},\XX](f,A)\le 1/t$. Then
\[
\osc[\XB,\XX](f,A)\le \frac{1}{K_\alpha t}.
\]
We infer that $\XB_{\alpha}$ is bounded-oscillation unconditional and
\[
\Phi[\XB_\alpha,\XX](t) \le\Phi[\XB,\XX](K_\alpha t), \quad 0<t\le 1.
\]
An application Lemma~\ref{lem:Philogconvex} puts an end to the proof.
\end{proof}

\begin{remark}\label{remarkstable}
Note that Lemma~\ref{lemma seminormalized1} entails that if $\XB$ is not unconditional, then for fixed $\alpha$ we have
\begin{align*}
\lim_{t\to 0} \max\enbrace{ \frac{\Phi[\XB_{\alpha},\XX](t)}{\Phi[\XB,\XX](t)}, \frac{\Phi[\XB,\XX](t)}{\Phi[\XB_{\alpha},\XX](t)}} =1.
\end{align*}
\end{remark}

We close this section with an application of our results to the study of the stability of quasi-greediness in $p$-Banach spaces.
It is known that quasi-greedy bases are stable with respect to products for semi-normalized scalar sequences (\cite{AABW2021}*{Theorem 4.11}, \cite{KoTe1999}, \cite{DKO2015}*{Lemma 2.1} \cite{Woj2000}*{Proposition 3}). Our results yield an alternative proof that improves the known estimates for the asymptotic growth of the constants. To contextualize this result, we need some more terminology. Given a basis $\XB$ in a $p$-Banach space $\XX$ and $0<t\le 1$, a set $A\subset_f\NN$ is said to be a \emph{$t$-greedy set} of $f\in\XX$ with respect to the basis $\XB$ if $\abs{\xx_n^*(f)} \ge t\abs{\xx_k^*(f)}$ for all $n\in A$ and $k\in\NN\setminus A$. Given $0<t\le 1$, the basis $\XB$ is said to be \emph{$t$-quasi-greedy} if there is a constant $C\ge 1$ such that $\norm{S_A(f)} \le C \norm{f}$ for all $f\in\XX$ and all $t$-greedy sets $A$ of $f$.

Note that if $\alpha$ is a semi-normalized scalar sequence, then any greedy set of $f$ with respect to $\XB_{\alpha}$ is a $K_{\alpha}$-greedy set of $f$ with respect to $\XB$, and, the other way around, any $t$-greedy set of $f$ with respect to $\XB$ is a greedy set of $f$ with respect to $\XB_{\beta}$ for some semi-normalized sequence $\beta$ with $K_{\beta}= t$. Thus, the questions of whether quasi-greedy bases are $t$-quasi-greedy for all $0<t<1$ and whether the bases $\XB_{\alpha}$ above are always quasi-greedy if $\XB$ is quasi-greedy are equivalent, and so are the results for the estimates of the constants involved.

To the best of our knowledge, the known estimates for the growth of the $t$-quasi-greedy constants grow as $1/t$
as $t$ tends to zero (see \cite{AABW2021}*{Theorem 4.11}, \cite{DKO2015}*{Lemma 2.1},\cite{KoTe1999},\cite{Woj2000}*{Proposition 3}). We can improve these estimates as follows:
\begin{lemma}\label{lemmastable}
Let $\XB$ be a quasi-greedy basis with constant $C_q$, $1\le C_q<\infty$, of a $p$-Banach space $\XX$. There is a constant $C_1$ depending only on $p$ and $C_q$ such that
\[
\norm{S_A(f)}\le C_1 (1-\log t)^{1/p} \norm{f}
\]
for every $f\in \XX$, $0<t<1$ and $A \subset_f \NN$ a $t$-greedy set of $f$.
\end{lemma}

\begin{proof}
By Remark~\ref{remark: QG->TQG}, $\XB$ is bounded-oscillation unconditional with $\Phi(1)\le 2^{1/p}C_q$. Let $C$ be the constant of Corollary~\ref{corollaryPhilog}, which only depends only on $p$ and $C_q$. Given $0<t<1$, $f\in \XX$ and $A\subset_f \NN$ a $t$-greedy set of $f$, let $B$ be a greedy set of $f$ with $\abs{B}=\abs{A}$. Since, in the case when $B\not=A$,
\[
\min_{n\in B\setminus A}\abs{\xx_n^*(f)}\ge\max_{n\in A\setminus B}\abs{\xx_n^*(f)} \ge \min_{n\in A\setminus B}\abs{\xx_n^*(f)} \ge t \max_{n\in B\setminus A}\abs{\xx_n^*(f)},
\]
it follows that $\osc{(f,(A\setminus B)\cup (B\setminus A))}\le 1/t$. Hence,
\begin{align*}
\norm{S_A(f)}^p&\le \norm{S_B(f)}^p+\norm{S_{(A\setminus B)\cup (B\setminus A)}f)}^p\\
&\le (C_q^p+C^p(1-\log t))\norm{f}^p. \qedhere
\end{align*}
\end{proof}

\section{Nearly unconditional bases that are not bounded-oscillation unconditional}\label{sect:EltonNotDilworth}\noindent
The goal of this section is to prove Theorem~\ref{thm:example}. We will do so with the help of some weakened forms of unconditionality that arise naturally in the study of the thresholding greedy algorithm. Recall that a basis $\XB=(\xx_n)_{n=1}^\infty$ is unconditional if and only if there is a constant $C\in[1,\infty)$ such that
\begin{equation}\label{eq:LU}
\norm{\sum_{n\in A} b_n\, \xx_n} \le C \norm{\sum_{n\in A} a_n\, \xx_n}
\end{equation}
for all $A\subset_f\NN$ and all families $(a_n)_{n\in A}$ and $(b_n)_{n\in A}$ with $\abs{b_n}\le \abs{a_n}$ for all $n\in A$. If this inequality holds for a given $C$, we say that $\XB$ is \emph{lattice unconditional} with constant $C$. The basis $\XB$ is is said to be \emph{lattice partially unconditional} (LPU for short) with constant $C$ if
\eqref{eq:LU} holds under the stronger assumption that
\[
\sup_{n\in A} \abs{b_n} \le \inf_{n\in A} \abs{a_n}.
\]
By convexity, $\XB$ is LPU with constant $C$ if and only if \eqref{eq:LU} holds in the case when
\[
\abs{b_n}=1\le \abs{a_n}, \quad n\in A.
\]
We call a basis $\XB$ \emph{lower unconditional for constant coefficients} (LUCC for short)
with constant $C$ if \eqref{eq:LU} is verified when
\[
\abs{b_n}=1\quad\text{and}\quad 1\le a_n/b_n<\infty,\quad n\in A.
\]
Finally, if \eqref{eq:LU} holds when $\abs{a_n}=1$ and either $b_n=a_n$ or $b_n=0$ for all $n\in A$, we say that $\XB$ is e \emph{unconditional for constant coefficients} (UCC for short).

A basis is LPU if and only if it is UCC and LUCC. Quantitatively, if $\ucc$, $\lucc$ and $\lpu$ are the optimal constants for unconditionality for constant coefficients, lower unconditionality for constant coefficients and lattice partial unconditionality, respectively,
\[
\max\{ \ucc,\lucc\} \le \lpu, \quad \lpu\le \Upsilon A_p \ucc \lucc,
\]
where $\Upsilon=2$ if $\FF=\RR$ and $\Upsilon=4$ if $\FF=\CC$ (see \cite{AABW2021}*{Section 3}). In turn, quasi-greedy bases for largest coefficients are UCC with $\ucc\le \qglc$, and being super-democratic is equivalent to being democratic and UCC.

Recall that a \emph{Schauder basis} of a quasi-Banach space $\XX$ is a sequence $\XB=(\xx_n)_{n=1}^\infty$ such that every $f\in\XX$ has a unique series expansion $f=\sum_{n=1}^\infty a_n\, \xx_n$, where the convergence of the series is in the topology induced by the quasi-norm. It is known that $\XB$ is a Schauder basis if and only if it spans the whole space $\XX$, there is a sequence $(\xx_n^*)_{n=1}^\infty$ in $\XX^{\ast}$ biorthogonal to $\XB$, and the partial sum projections associated to the basis, given by
\[
S_m=S_{\NN^{\le m}}, \quad \text{where}\;\;\NN^{\le m}=\{1,\dots,m\},
\]
are uniformly bounded on $m\in\NN$. Hence, any semi-normalized Schauder basis of $\XX$ is a basis in the sense we set up at the beginning. The number
\[
\sch=\sup_m \norm{S_m}
\]
is the \emph{basis constant} and if $\sch=1$ we say that $\XB$ is \emph{monotone}. If a Schauder basis satisfies the stronger condition $ \norm{S_m-S_k}\le 1$ for all $1\le k \le m<\infty$, we say that it is \emph{bimonotone}.

The proof of Theorem~\ref{thm:example} will be a ready consequence of Theorem~\ref{thm:Miguel} below. To understand the setting of the proof, the reader should be aware that it makes sense to consider bases indexed over a countable (finite or infinite) set, and all unconditionality-like properties we have introduced can be studied with this convention in mind. Also, we can consider Schauder bases indexed over a countable set endowed with an order which is isomorphic to $\NN$ or to $\NN^{\le m}$ for some $m\in\NN$. Needless to say, the interest in the study of finite bases lies in obtaining estimates for the constants involved.

Suppose that for $M\in\NN$ we let $\XB_M= (\xx_{M,n})_{n\in\Nt_M}$ be a finite family of vectors in a quasi-Banach space $\XX_M$. Set
\[
\Nt=\bigcup_{M=1}^\infty \{M\}\times \Nt_M
\]
and denote by $J_M$ the canonical embedding of $\XX_M$ into $\prod_{M=1}^\infty \XX_M$. The direct sum of $(\XB_M)_{M=1}^\infty$ is the family
\[
\XB:=\oplus_{M=1}^\infty \XB_M=(\yy_\nu)_{\nu\in\Nt}
\]
defined by $\yy_{M,n}=J_M(\xx_{M,n})$ for all $(M,n)\in\Nt$. If $\XB_M$ is a monotone normalized Schauder basis of $\XX_M$ for each $M\in\NN$ and we consider the lexicographical order on $\Nt$, we see that $\XB$ is a monotone normalized Schauder basis of $(\oplus_{M=1}^\infty \XX_M)_p$ for all $1\le p<\infty$. As far as the unconditionality-like conditions we are studying is concerned, $\XB$ is truncation quasi-greedy (or QGLC, or LPU, or LUCC, or UCC) with constant $C$ if and only if each summand is. As for the democracy of $\XB$, we point out that it
inherits the estimates
\begin{equation}\label{eq:FFlpBis}
\frac{1}{C} \abs{A}^{1/p} \le \norm{\Ind_{\varepsilon,A}}[\XB_M,\XX_M]\le C \abs{A}^{1/p}, \quad A\subset \Nt_M,
\end{equation}
from its summands, where $C$ is a constant that does not depend on $M$.

Given $p\in[1,\infty]$, let $p'$ denote its conjugate exponent defined by $1/p+1/p'=1$.

Given two subsets $A$ and $B$ of $\RR$, the symbol $A<B$ means that $a<b$ for all $a\in A$ and $b\in B$.

\begin{theorem}\label{thm:Miguel}
Given $1\le p<\infty$, there is Banach space $\XX$ with a normalized Schauder basis $\XB$ such that
\begin{enumerate}[label=(\roman*), leftmargin=*, widest=iii]
\item $\XB$ is QGLC and LUCC, hence LPU,
\item $\XB$ is democratic with $\udf(m)\approx m^{1/p}$ for $m\in\NN$, and
\item $\XB$ is not truncation quasi-greedy.
\end{enumerate}
\end{theorem}

\begin{proof}
It suffices to prove that there is a positive constant $C=C(p)$ such that, for every $M\in \NN$, there is a set $ E_M\subset_f \NN$ and a finite-dimensional normed space $\XX_M$ with a monotone normalized Schauder basis $\XB_M=(\xx_{M,n})_{n\in E_M}$ satisfying \eqref{eq:FFlpBis} whose truncation quasi-greedy constant, quasi-greedy for largest coefficients constant, and partial unconditionality constant satisfy $\tqg> M$ and $\max\{\qglc,\lpu\}\le C$. To show that such finite-dimensional bases exists we need to feed a rather involved construction with a basis $\YB$ of a Banach space $(\YY, \norm{\cdot}_\YY)$ and a nonincreasing sequence $(c_n)_{n=1}^\infty$ in $(0,1/4p]$ satisfying the following properties for some constant $K$:
\begin{enumerate}[label=(P.\arabic*), leftmargin=*, widest=3]
\item $\YB$ is a bimonotone normalized Schauder basis with
\begin{equation}\label{lemmaqglclpanotqgdem2}
\frac{1}{K}
\norm{\Ind_{\varepsilon, D}}_{\YY}\le \abs{D}^{1/p}\le K \norm{\Ind_{\varepsilon, D}}_{\YY}
\end{equation}
for every finite set $D\subset \NN$ and every $\varepsilon\in \EE^D$. Besides, $\YB$ is truncation quasi-greedy with constant $K$, and LPU with constant $K$.

\item\label{cndos} $\norm{(c_n)_{n=1}^\infty}_{\ell_p}=\infty$.

\item\label{cn} There is $y\in\YY$ whose coefficients with respect to $\YB$ are $(c_n)_{n=1}^\infty$, so that
$M_0:=1+\sup_{1\le k\le m}\norm{\sum_{n=k}^{m}c_n \, \yy_n}_{\YY}<\infty$.
\end{enumerate}

Notice that if an unconditional basis $\YB_0$ whose fundamental function is equivalent to the fundamental function of the unit vector basis in $\ell_p$
is dominated (but not equivalent) to the canonical $\ell_p$-basis, then there are scalars $(c_n)_{n=1}^\infty$ and a rearrangement $\YB$ of $\YB_0$ such that these conditions are satisfied. So if $p=1$ we could pick for instance the unit vector system of Tsirelson's space, or the natural basis of the dyadic Hardy space $H_1$. If $p>1$ one may choose the canonical unit vector basis of the weak Lorentz space $\ell_{p, \infty}$ (which is a Banach space under a suitable renorming). Note that, since $\ell_{p,1}\subset\ell_p$, \ref{cndos} implies $\sum_{n=1}^\infty n^{-1/p'} c_n=\infty$.

We start our construction picking $m_1\in\NN^{\ge 4}$ such that
\begin{equation}\label{lemmaqglclpanotqg cond0}
\sum_{n=1}^{m_1}c_n^p > (10 M^2M_0)^p.
\end{equation}
Set $A=\{1,\dots,m_1\}$ and $M_2=5/c_{m_1}$. Choose integer intervals $(I_n)_{n=1}^{m_1}$ and $(B_n)_{n=1}^{m_1}$ so that
\begin{enumerate}[label=(\arabic*), leftmargin=*, widest=3]
\item $A<I_n<B_n<I_{n+1}<B_{n+1}$ for all $1\le n\le m_1-1$, and
\item for all $n\in A$,
\begin{align}\label{lemmaqglclpanotqg condII}
2M_2c_n&\le \sum_{k\in I_n} \frac{c_k}{k^{1/p'}} \le 3M_2c_n,\\
\sum_{k=1}^{\abs{B_n}} \frac{1}{k^{1/p'}}\le M_2 c_n &\le \sum_{k=1}^{\abs{B_n}+1} \frac{1}{k^{1/p'}}\label{lemmaqglclpanotqg condIII}
\end{align}
\end{enumerate}

Let $I=\bigcup_{n\in A}I_n$, $B:=\bigcup_{n\in A}B_n$, and $E_M=A\cup B\cup I$. Let $l_1$ be the largest integer in $E_M$, i.e., $l_1=\max B_{m_1}$. For each $n\in A$, put $\C_n=\{n\}\cup I_n\cup B_n$, and let $\mathcal{A}$ be the set of sequences $\Delta=(\delta_n)_{n\in E_M}$ with the following properties:
\begin{enumerate}[label=(R.\arabic*), leftmargin=*, widest=3]
\item For each $n\in E_M$, $\abs{\delta_n}\le 1$.
\item For each $n\in A$,
\[
M_2\, c_n \delta_n +\sum_{k\in I_n}\delta_k\, c_k +\sum_{k\in B_n}\delta_k=0.
\]
\item\label{qglclpunotqgI} For each $k\in I$, $\abs{\delta_k}\le k^{-{1}/{p'}}$.
\item \label{qglclpunotqgB}For each $n\in A$ and $D\subset B_n$,
\begin{align*}
\sum_{j\in D}\abs{\delta_j}\le& \sum_{k=1}^{\abs{D}}\frac{1}{k^{{1}/{p'}}}.
\end{align*}
\end{enumerate}

For each $\Delta \in \mathcal{A}$, $n\in A$ and $l \in \C_n$ we define a semi-norm on $\XX_M:=\FF^{E_M}$ by
\[
\norm{f}_{\Delta,n,l}=
\left(\sum_{1\le j< n}\left\vert \delta_j\, a_j+\sum_{k\in I_j}\delta_k \, a_k +\sum_{k\in B_j}\delta_k \, a_k\right\vert^p+\left\vert \sum_{\substack{ j\in \C_n\\n\le j\le l}}\delta_j a_j\right\vert^p\right)^{{1}/{p}},
\]
where $f=(a_n)_{n\in E_M}$.

Now, we define a seminorm $\norm{\cdot}_{\triangleleft}$ on $\XX_M$ by
\[
\norm{f}_{\triangleleft} =\sup_{\Delta\in \mathcal{A}}\sup_{n\in A}\sup_{l\in \C_n} \norm{f}_{\Delta,n,l},
\]
and we define a norm on $\XX_M$ by
\[
\norm{ f}_{\XX_M}=\max\enbrace{ \norm{f}_{\ell_{\infty}}, \norm{ S_A(f)}_{\YY}, \norm{S_I(f)}_{\YY}, \norm{ f}_{\triangleleft} }.
\]

Equip $E_M$ with the ordering $\preceq$ defined as
\begin{itemize}
\item If $n_1, n_2\in A$, $n_1<n_2$ and $k\in I_{n_1}\cup B_{n_1}$, then $n_1\prec k\prec n_2$,
\item If $n_1,n_2\in I\cup B$, then $n_1\preceq n_2$ if and only if $n_1\le n_2$.
\end{itemize}
Let $\XB_M=(\xx_{M,n})_{n\in E_M}$ be the canonical basis of $\XX_M$. Since $\YB$ is bimonotone and normalized,
$\XB_M$ is a monotone normalized Schauder basis of $(\XX_M,\norm{\cdot}_{\XX_M})$ with normalized dual basis. Let us check that it sastisfies the desired properties by finding an upper bound for the fundamenal function of $\XB_M$.

Fix $D\subset E_M$ and $\varepsilon=(\varepsilon_j)_{j\in D}\in \EE^D$. Set $D_1:=D\cap A$, $D_2:=D\cap I$, and $D_3=D\cap B$. For $\Delta\in \mathcal{A}$, $n\in A$ and $l\in \C_n$, we have
\begin{align*}
\norm{\Ind_{\varepsilon, D_1}}_{\Delta,n,l}^p&=\sum_{1\le j< n}\abs{ \delta_j \, \xx_{M,j}^*\left(\Ind_{\varepsilon, D_1}\right) }^p+\abs{ \delta_n\, \xx_{M,n}^*\left(\Ind_{\varepsilon, D_1}\right)}^p
\le \abs{D_1},\\
\norm{\Ind_{\varepsilon, D_2}}_{\Delta,n,l}^p
&\le \sum_{j\in A}\left(\sum_{k\in I_j\cap D}\abs{\delta_k }\right)^p \le \sum_{j\in A}\left(\sum_{k=1}^{\abs{I_j\cap D}}\frac{1}{k^{{1}/{p'}}}\right)^p
\le p^p \abs{D_2},\\
\norm{\Ind_{\varepsilon, D_3}}_{\Delta,n,l}^p&\le \sum_{j\in A} \left(\sum_{k\in B_j\cap D_3}\abs{\delta_k} \right)^p
\le \sum_{j\in A} \left(\sum_{k=1}^{\abs{B_j\cap D_3}}\frac{1}{k^{{1}/{p'}}}\right)^p \le p^p \abs{D_3},
\end{align*}
so that
\[
\norm{\Ind_{\varepsilon, D}}_{\triangleleft} \le 3^{1/p} p \abs{D}^{1/p}.
\]
Hence, using \eqref{lemmaqglclpanotqgdem2},
\begin{equation}\label{eq:upperdemocracy}
\norm{\Ind_{\varepsilon, D}}_{\XX_M} \le C_1 \abs{D}^{1/p}, \, \mbox{ where }C_1= \max\{3^{1/p} p, K\}.
\end{equation}

In order to estimate the QGLC and the LPU constant of $\XB_M$, we fix $f\in \Cu$ with $\supp(f)\cap D=\emptyset$ and $g=(a_n)_{n\in E_M}\in\XX_M$ such that $\abs{a_n}\ge 1$ if $n\in D$ and $a_n= 0$ otherwise. Let $h$ be either the vector $\Ind_{\varepsilon, D}+f$ or the vector $g$. Since $\YB$ is QGLC and LPU with constant $K$, and \eqref{lemmaqglclpanotqgdem2} holds, we obtain
\[
\frac{1}{K} \abs{D_1}^{{1}/{p}}\le
\norm{ \Ind_{\varepsilon, D_1}}_{\YY}=\norm{ S_A\left(\Ind_{\varepsilon, D}\right)}_{\YY}
\le K \norm{ S_A(h)}_{\YY}
\le K \norm{h}_{\XX_M}.
\]
Similarly,
\[
\frac{1}{K} \abs{D_2}^{{1}/{p}}
\le \norm{ \Ind_{\varepsilon, D_2}}_{\YY}
=\norm{ S_I\left(\Ind_{\varepsilon, D}\right)}_{\YY}\le K \norm{ S_I(h)}_{\YY}
\le K \norm{h}_{\XX_M}.
\]

With the aim to obtain a similar estimate for the norm of $\Ind_{\varepsilon, D}+f$ which involves $\abs{D_3}$, we set $A_3=\{n\in A \colon D\cap B_n\not=\emptyset\}$. For each $n\in A_3$, let $(k_j^n)_{j=1}^{\abs{D\cap B_n}}$ be the increasing enumeration of $D\cap B_n$. Put
\[
\delta_{k_j^n}= \frac{1}{j^{{1}/{p'}}\varepsilon_{k_j^n}}, \quad j=1,\dots, \abs{D\cap B_n},
\]
and define
\[
\delta_n:=-\frac{1}{M_2c_n}\sum_{j=1}^{|D\cap B_n|}\delta_{k_j^n}, \quad n\in A_3.
\]

For all the other $n\in E_M$, let $\delta_n=0$. Then by construction $\Delta=(\delta_n)_{n\in E_M}\in \mathcal{A}$. Since $M_2c_n>2$ for all $n\in A$,
\begin{align*}
\norm{ \Ind_{\varepsilon, D}+f}_{\XX_M}^p &\ge \norm{ \Ind_{\varepsilon, D}+f}_{\Delta,m_1,l_1}^p \nonumber\\
&= \sum_{n\in A_3}\abs{\delta_n\xx_n^*\left(\Ind_{\varepsilon, D}+f\right)+\sum_{k\in B_n}\delta_k \xx_k^*\left(\Ind_{\varepsilon, D}+f\right)}^p\\
&\ge \sum_{n\in A_3}\abs{ -\abs{ \delta_n} +\sum_{k=1}^{\abs{D\cap B_n}}\frac{1}{k^{{1}/{p'}}} }^p\\
&\ge \frac{1}{2^p}\sum_{n\in A_3}\left( \sum_{k=1}^{\abs{D\cap B_n}}\frac{1}{k^{{1}/{p'}}} \right)^p\\
&\ge \frac{1}{(4p)^p}\abs{D_3}.
\end{align*}
Summing up,
\begin{equation}\label{eq:QGLC}
\abs{D}^{1/p} \le C_2 \norm{ \Ind_{\varepsilon, D}+f}_{\XX_M}, \, \mbox{ where }C_2= \left( (4p)^p+2K^{2p}\right)^{1/p}.
\end{equation}
In the particular case that $f=0$, inequality \eqref{eq:QGLC} yields
\begin{equation}\label{eq:lowerdemocracy}
\abs{D}^{1/p} \le C_2 \norm{ \Ind_{\varepsilon, D}}_{\XX_M}.
\end{equation}

Recall that we have related the norm of $g$ with $\abs{D_1}$ and $\abs{D_2}$. Now we take care of relating the norm of $g$ with $\abs{D_3}$. To that end we build another family $(\delta_n)_{n\in E_M}$. Put
\[
A_{3,1}=\enbrace{ n\in A_3 \colon \sum_{k\in I_n\setminus D}\frac{c_k}{k^{{1}/{p'}}}\ge \sum_{k=1}^{\abs{B_n}}\frac{1}{k^{{1}/{p'}}} }
\]
and
\[
D_{3,1}=\bigcup_{n\in A_{3,1}}D\cap B_n.
\]

For each $n\in A_{3,1}$, let
\[
\delta_{k^n_j}=\frac{1}{j^{{1}/{p'}}\sgn\left(a_{k_j^n}\right)}, \quad j=1,\dots, \abs{B_n\cap D}.
\]
Then we choose scalars $(\delta_j)_{j\in I_n\setminus D}$ with $\abs{\delta_j}\le j^{-{1}/{p'}}$ for each $j\in I_n\setminus D$ such that
\[
\sum_{j\in I_n\setminus D}\delta_j c_j+\sum_{j=1}^{|B_n\cap D|}\delta_{k_j^n}=0.
\]
For any other $j\in E_M$, set $\delta_j=0$. Then $\Delta=(\delta_j)_{j\in E_M}\in \mathcal{A}$, and
\begin{align*}
\norm{g}_{\XX_M}\ge \norm{g}_{\Delta,m_1,l_1}
&=\left(\sum_{n\in A_{3,1}}\left(\sum_{j=1}^{\abs{B_n\cap D}}\frac{\abs{a_{k_j^n}}}{ j^{{1}/{p'}}}\right)^p\right)^{1/p}\\
&\ge\sum_{n\in A_{3,1}}\sum_{j=1}^{\abs{B_n\cap D}}\frac{1}{ j^{{1}/{p'}}}\\
&\ge \frac{1}{2p}\abs{ D_{3,1}}^{{1}/{p}}.
\end{align*}
It remains to estimate the cardinality of
\[
D_{3,2}=D_3\setminus D_{3,1}=\bigcup_{n\in A_3\setminus A_{3,1}}D\cap B_n.
\]
From \eqref{lemmaqglclpanotqg condII}, \eqref{lemmaqglclpanotqg condIII}, and the fact that $\abs{c_j}\le 1/(4p)$ for all $j\in \NN$, it follows that for each $n\in A_{3,2}:=A\setminus A_{3,1}$,
\begin{align*}
\abs{D\cap I_n}^{{1}/{p}}\ge 4\sum_{k\in I_n\cap D}\frac{c_k}{k^{{1}/{p'}}}&=4\left(\sum_{k\in I_n}\frac{c_k}{k^{{1}/{p'}}}-\sum_{k\in I_n\setminus D}\frac{c_k}{k^{{1}/{p'}}}\right)\\
&\ge 4\left( \sum_{k\in I_n}\frac{c_k}{k^{{1}/{p'}}}-\sum_{k=1}^{\abs{B_n}}\frac{1}{k^{\frac{1}{p'}}}\right)\\
&\ge 4\sum_{k=1}^{\abs{B_n}}\frac{1}{k^{{1}/{p'}}}\\
&\ge \frac{2}{p}\abs{B_n}^{{1}/{p}}.
\end{align*}
Hence,
\[
\abs{D_{3,2}}=\sum_{n\in A_{3,2}}\abs{D\cap B_n}\le \sum_{n\in A_{3,2}}p^p \abs{D\cap I_n}\le p^{p}\abs{D_{2}}.
\]
Summing up,
\begin{equation}\label{eq:LPU}
\abs{D}^{1/p} \le \left( \abs{D_1} +\abs{D_{3,1}} + (1+p^p) \abs{D_2}\right)^{1/p} \le C_3 \norm{g}_\XX,
\end{equation}
where
\[C_3= ((2+p^p)K^{2p}+(2p)^p)^{1/p}.
\]

Combining \eqref{eq:upperdemocracy}, \eqref{eq:QGLC}, \eqref{eq:lowerdemocracy}, and \eqref{eq:LPU} yields the desired estimates for $\XB_M$ with $C=
C_1\max\{ C_2, C_3\}$.

Finally, we prove the lower estimate for the truncation quasi-greedy constant.
Set
\[
f_0=\Ind_{B}+\sum_{n\in A}M_2c_n\xx_{M,n}+\sum_{k\in I}c_k\xx_{M,k}.
\]
From \eqref{lemmaqglclpanotqg condII}, \eqref{lemmaqglclpanotqg condIII}, and the definition of $\mathcal{A}$ it follows that
\begin{align*}
\norm{f_0}_{\triangleleft}&=\sup_{\Delta\in\mathcal{A}}\sup_{n\in A}\sup_{l\in \{n\}\cup B_n}\abs{\sum_{\substack{j\in \C_n\\n\le j\le l}}\delta_j\xx_j^*(f_0)}\\
&\le \sup_{\Delta\in\mathcal{A}}\sup_{n\in A}\left(M_2c_n+\sum_{k=1}^{\abs{B_n}}\frac{1}{k^{{1}/{p'}}}+\sum_{k\in I_n}\frac{c_k}{k^{{1}/{p'}}}\right)\\
&\le 5M_2c_1< 5M_2.
\end{align*}

Also, $\norm{ S_A(f_0)}_{\YY}\le M_2M_0$ and $\norm{ S_I(f_0)}_{\YY}\le M_0$ by definition of $M_0$, so combining we get $\norm{ f_0}_{\XX_M}\le 5M_2M_0$.

In order to estimate from below the norm of $\Ind_B$, for each $n\in A$ define $(\delta_k)_{k\in B_n}$ by $\delta_{k_{n,j}}:=j^{-{1}/{p'}}$, where
$(k_{n,j})_{j=1}^{\abs{B_n}}$ is the increasing rearrangement of $B_n$. Set
\[
\delta_n=-(M_2c_n)^{-1}\sum_{j=1}^{\abs{B_n}}\frac{1}{j^{{1}/{p'}}},\quad n\in A,
\]
and let $\delta_k=0$ for every other $k\in E_M$. Then $\Delta=(\delta_k)_{k\in \XX_M}\in \mathcal{A}$. Considering our choice of $A$ and \eqref{lemmaqglclpanotqg condIII} we deduce that
\begin{align*}
\norm{\Ind_B}_{\XX_M}\ge \norm{\Ind_B}_{\Delta,m_1,l_1}
&= \left(\sum_{n\in A} \left(\sum_{j=1}^{\abs{B_n}}\frac{1}{j^{{1}/{p'}}}\right)^p\right)^{{1}/{p}}\\
&\ge \left(\sum_{n\in A}\left( \frac{M_2c_n}{2}\right)^p\right)^{{1}/{p}}\\
&> 5M_2 M_0M^2.
\end{align*}
Since
\[
\norm{ \Ind_B}_{\XX_M}\le \tqg\norm{\Ind_{A\cup B}}_{\XX_M} \le \tqg^2 \norm{ f_0}_{\XX_M},
\]
the above estimates combined yield that $\tqg>M$, and the proof is complete.
\end{proof}

\begin{proof}[Proof of Theorem~\ref{thm:example}]
Just combine Theorem~\ref{thm:Miguel}, Theorem~\ref{thm:main}, and Theorem~\ref{thm:AAB2023}.
\end{proof}

\section{The role of the truncation operators}\label{TOvsQGLC}\noindent
Besides showing that truncation quasi-greediness is a stronger property than quasi-greediness for largest coefficients, Theorem~\ref{thm:Miguel} yields the first known example of a LPU basis that is not truncation quasi-greedy. In this section we focus on the relation between QGLC and LUCC bases. This study is tightly connected with a better understanding of the behavior of the truncation operators.

Given a basis $\XB$ of a quasi-Banach space and $A\subset_f\NN$, for $f\in\XX$ we set
\begin{align*}
\RTO_A(f) &= \min_{n\in A} \abs{\xx_n^*(f)} \Ind_{\varepsilon(f),A}, \\
\TO_A(f)&=\RTO_A(f)+f-S_A(f).
\end{align*}
Using this notation, a basis $\XB$ is truncation quasi-greedy if and only if the \emph{restricted truncation operators} $\RTO_A$ have the property that $\RTO_A(f)$ remains bounded when $f$ runs over the unit ball of $\XX$ and $A$ runs over all greedy sets of $f$. From now on we will use the term \emph{flattening quasi-greedy bases} to refer to those bases whose \emph{truncation operators} $\TO_A$ have the corresponding property, and define the flattening quasi-greedy constant $\fqg$ as the optimal constant $C$ such that for all greedy sets $A$ of $f\in\XX$,
\[
\norm{\TO_A(f)}\le C \norm{f}.
\]
Despite the fact that the evolution of the theory has recently evinced the relevance of the restricted truncation operators, the informed reader will be surely aware that the truncation operators made their appearance in abstract greedy approximation theory way back in the important article \cite{DKKT2003}. If $\XB$ is quasi-greedy, then apart from truncation quasi-greedy, $\XB$ is also flattening quasi-greedy. Moreover, the truncation operators play an important role in the isometric theory of greedy-like bases \cite{AlbiacAnsorena2017b}. Since flattening quasi-greedy bases are LUCC, the following result shows that LUCC bases need not be QGLC.

\begin{theorem}\label{thm:FTQvsUCC}
There is a monotone Schauder basis of a Banach space that it is flattening quasi-greedy but it is not UCC.
\end{theorem}

The main technique to prove Theorem~\ref{thm:FTQvsUCC} depends on knowing how to combine unconditional bases to obtain flattening quasi-greedy bases in a way that fits our purposes. Let us next detail the construction we will carry out.

Given Banach spaces $\XX$ and $\YY$, we will endow its direct sum $\XX\oplus\YY$ with the suppremum norm
\[
\norm{(x,y)}=\max\enbrace{\norm{x},\norm{y}}, \quad x\in\XX, \, y\in\YY,
\]
and we will identify $\XX^*\oplus\YY^*$ with the dual space of $\XX\oplus\YY$. We will denote by $P_{\XX}$ and $P_{\YY}$ the coordinate proyections from $\XX\oplus\YY$ onto $\XX$ and $\YY$, respectively.

Given sequences $\XB=(\xx_n)_{n=1}^\infty$ and $\YB=(\yy_n)_{n=1}^\infty$ in $\XX$ and $\YY$, respectively, we consider the sequences $\XB \ltimes \YB=(\zz_n)_{n=1}^\infty$ and $\XB \rtimes \YB=(\ww_n)_{n=1}^\infty$ in $\XX\oplus \YY$ defined for each $k\in\NN$ by
\[
\zz_{2k-1}=(\xx_{k},0),\, \zz_{2k}=(-\xx_k,\yy_{k}), \quad \ww_{2k-1}=(\xx_k,\yy_k), \, \ww_{2k}=(0,\yy_k).
\]

Theorem~\ref{thm:FTQvsUCC} will be a consequence of Lemma~\ref{lemmacomb2}. For further reference, prior to state it we record an elementary lemma.

\begin{lemma}\label{lemmacomplex}
Let $z_1$ and $z_2$ be complex escalars with $\abs{z_1}\le \abs{z_2}$, and let $1\le t_1 \le t_2<\infty$. Then
\[
\abs{z_1-z_2} \le \abs{t_1 \, z_1-t_2 \, z_2}.
\]
\end{lemma}
\begin{proof}
It suffices to check that the the function
\[
f\colon\RR\to[0,\infty), \quad t\mapsto f(t)=\abs{t\varepsilon_1+\varepsilon_2}
\]
is increasing on $[1,\infty)$. This is a consequence of the fact that $f$ is a parabolic function with vertex at the point $\Re(\varepsilon_1/\varepsilon_2)\in[-1,1]$.
\end{proof}

\begin{lemma}\label{lemmacomb2}Let $\XB=(\xx_n)_{n=1}^\infty$ and $\YB=(\yy_n)_{n=1}^\infty$ be bases of Banach spaces $\XX$
and $\YY$ respectively, and let $\XB^*=(\xx_n^*)_{n=1}^\infty$ and $\YB^*=(\yy_n^*)_{n=1}^\infty$ be their dual bases. Then $\XB \ltimes \YB=(\zz_n)_{n=1}^\infty$ is a basis of $\XX\oplus\YY$ with dual basis $\XB^* \rtimes \YB^*=(\zz_n^*)_{n=1}^\infty$. The coordinate projections on $\XX$ and $\YY$ are determined by the identities
\begin{equation}\label{eq:BeraSeis}
\xx_n^*\circ P_\XX = \zz_{2n-1}^*-\zz_{2n}^*,\quad \yy_n^*\circ P_\YY=\zz_{2n}^*, \quad n\in \NN.
\end{equation}
Moreover,
\begin{enumerate}[label=(\roman*), leftmargin=*, widest=ii]
\item\label{lemmacomb2:Sch} if $\XB$ and $\YB$ are Schauder bases, so is $\XB \ltimes \YB$, and
\item\label{lemmacomb2:FQG} if $\XB$ and $\YB$ are unconditional bases, then $\XB \ltimes \YB$ is flattening quasi-greedy.
\end{enumerate}
\end{lemma}
\begin{proof}
It is routine to check that $\XB\ltimes\YB$ is a basis and that \eqref{eq:BeraSeis} holds. To prove \ref{lemmacomb2:Sch} we note that, by \eqref{eq:BeraSeis},
the partial sums projections $(S_m)_{m=1}^\infty$, $(S_m^x)_{m=1}^\infty$ and $(S_m^y)_{m=1}^\infty$ relative to $\XB\ltimes\YB$, $\XB$ and $\YB$ respectively are related by the identity
\[
S_{2m}=(S_m^x\circ P_\XX, S_m^y\circ P_\YY), \quad m\in\NN.
\]
Let $K_1$ and $K_2$ the bases constants of $\XB$ and $\YB$, respectively. Set $\beta=\inf_n \norm{\yy_n}$ and $\alpha=\sup_{n\in \NN}\norm{\xx_n}$. We have
\begin{align*}
\sup_{m\in\NN} \norm{S_m}
&\le \max_{m\in \NN} \norm{\zz_{2m-1}} \norm{\zz_{2m-1}^*} + \sup_{m\in\NN} \norm{S_{2m}}\\
&= \sup_{n\in\NN} \norm{\xx_n} \max\{ \norm{\xx_n^*}, \norm{ \yy_n^*}\}
+\sup_{m\in\NN}\max\{\norm{S_m^x}, \norm{S_m^y}\}\\
& \le \sup_{n\in\NN} \max\enbrace{ 2K_1, \frac{\norm{\xx_n}}{\norm{ \yy_n}_\YY} 2K_2}
+\max\{K_1, K_2\}\\
&\le \enpar{1+ 2\max\enbrace{1,\frac{\alpha}{\beta}}}\max\{K_1,K_2\}.
\end{align*}

To see \ref{lemmacomb2:FQG} we pick functions $f$ and $g$ in $\XX\oplus \YY$ for which there is $A\subset \NN$ finite such that
\begin{itemize}
\item $a_n:=\zz_n^*(f)=\zz_n^*(g)$ and $\abs{a_n}\le 1$ for all $n\in\NN\setminus A$, and
\item $\varepsilon_n:=\zz_n^*(f)\in\EE$ and $t_n:=\zz_n^*(g)/\zz_n^*(f)\ge 1$ for all $n\in A$.
\end{itemize}
Suppose that $\YB$ is lattice-unconditional with constant $C_2$. The right-hand identity in \eqref{eq:BeraSeis} entails that
\begin{equation}\label{eq:Bera5}
\norm{ P_{\YY}(f)}\le C_2\norm{ P_{\YY}(g)}.
\end{equation}

In order to compare the norms of $P_\XX(f)$ and $P_\XX(g)$, the left-hand identity in \eqref{eq:BeraSeis} impels us to split $\NN$ into the pairwise disjoint sets
\begin{align*}
A_1&=\{n\in \NN \colon 2n-1\in A,\, 2n\in \NN\setminus A\},\\
A_2&=\{n\in \NN \colon 2n-1\in \NN\setminus A, \, 2n \in A\},\\
A_3&=\{n\in \NN \colon 2n-1\in A, \, 2n \in A\},\\
A_4&=\{n\in \NN \colon 2n-1\in \NN\setminus A,\, 2n \in \NN\setminus A\},
\end{align*}
and express $ f_n:=\xx_n^*(P_\XX(f))$ and $g_n:=\xx_n^*(P_\XX(g))$, $n\in\NN$, according to this partition. We have
\[
(f_n,g_n)=
\begin{cases}
(\varepsilon_{2n-1}-a_{2n}, t_{2n-1}\varepsilon_{2n-1}-a_{2n}) & \mbox{ if } n\in A_1,\\
( a_{2n-1}-\varepsilon_{2n}, a_{2n-1}- t_{2n}\varepsilon_{2n}) & \mbox{ if } n\in A_2,\\
(\varepsilon_{2n-1}-\varepsilon_{2n}, t_{2n-1}\varepsilon_{2n-1}- t_{2n} \varepsilon_{2n}) & \mbox{ if } n\in A_3,\\
(a_{2n-1}-a_{2n},a_{2n-1}-a_{2n}) & \mbox{ if } n\in A_4.
\end{cases}
\]
We infer from lemma~\ref{lemmacomplex} that $\abs{f_n}\le \abs{g_n}$ for all $n\in\NN$.
Then, if $\XB$ is lattice unconditional with constant $C_1$,
\begin{equation}\label{eq:Bera6}
\norm{ P_{\XX}(f)}\le C_1\norm{ P_{\XX}(g)}.
\end{equation}
Combining \eqref{eq:Bera5} and \eqref{eq:Bera6} the proof is complete.
\end{proof}

\begin{proof}[Proof of Theorem~\ref{thm:FTQvsUCC}]
Take two unconditional normalized bases $\XB$ and $\YB$ of Banach spaces $\XX$ and $\YY$ respectively, with the property that there is a sequence $(A_m)_{m\in \NN}\subset \NN$ consisting of finite subsets of $\NN$ such that
\[
1\le \frac{\alpha_m}{\beta_m}\xrightarrow[m\to \infty]{}\infty,
\]
\mbox{ where }
\[\alpha_m=\norm{\Ind_{A_m}[\XB,\XX]} \mbox{ and } \beta_m=\norm{\Ind_{A_m}[\YB,\YY]}.\]
Now combine the bases using Lemma~\ref{lemmacomb2} to obtain a flattening quasi-greedy Schauder basis $\ZB:=\XB\ltimes \YB$ of $\UU:=\XX\oplus \YY$. Define
\[
B_m=\{2n-1 \colon n\in A_m\}, \quad D_m=\{2n \colon n\in A_m\}, \quad m\in\NN.
\]
By \eqref{eq:BeraSeis} in Lemma~\ref{lemmacomb2},
\[
\Ind_{B_m}[\ZB,\UU]=(\Ind_{A_m}[\XB,\XX], 0), \quad \Ind_{D_m}[\ZB,\UU]=(-\Ind_{A_m}[\XB,\XX], \Ind_{B_m}[\YB,\YY]).
\]
Consequently, $\norm{\Ind_{B_m}[\ZB,\UU]}= \alpha_m$ and $\norm{\Ind_{B_m\cup D_m}[\ZB,\UU]}= \beta_m$, which proves that $\ZB$ is not SUCC.
\end{proof}

In a sense, Theorem~\ref{thm:FTQvsUCC} evinces that, despite the important role played by the truncation operators in the study of the thresholding greedy algorithm, flattening quasi-greediness and lower uconditionality for constant coefficients can barely be considered unconditionality-like properties of interest in the study of the efficiency of the TGA. In this regard, we show that LUCC bases need not be QGLC even under the extra assumption that they are UCC.

\begin{theorem}\label{thm:LPUnotQGLC}
There is a monotone Schauder basis $\XB$ of a Banach space $\XX$ that it is LUCC and super-democratic, hence LPU, but is not QGLC. Moreover, given $1\le p<\infty$ we can choose $\XB$ to satisfy $\udf(m)\approx m^{1/p}$ for $m\in\NN$.
\end{theorem}

\begin{proof}
Our proof has a similar flavour to that of Theorem~\ref{thm:Miguel}. As in there, it suffices to show that there is a positive constant $C=C(p)$ such that, for every $M\in \NN$, there is a set $ E_M\subset_f \NN$ and a finite-dimensional normed space $\XX_M$ with a monotone normalized Schauder basis $\XB_M=(\xx_{M,n})_{n\in E_M}$ satisfying \eqref{eq:FFlpBis} whose quasi-greedy for largest coefficients constant and partial unconditionality constant satisfy $\qglc> M$ and $\lpu\le C$.

Let $\YB$, $(c_n)_{n=1}^\infty$, $M_0$ and $K$ be as in the proof of Theorem~\ref{thm:Miguel}. Let $C=C(p,K)$ be a constant that will be made explicit later, and choose an integer $m_1>(3M C M_0)^p$. As in the aforementioned proof, set $A=\{1,\dots,m_1\}$. Now choose integer intervals $(I_n)_{n\in A}$ so that:

\begin{enumerate}[label=(\arabic*), leftmargin=*, widest=2]
\item $A<I_n<I_{n+1}$ for $n=1$, \dots, $m_1-1$, and
\item
for all $n\in A$,
\begin{equation}\label{lemmalpunoqglc condII}
1<\sum_{k\in I_n}\frac{c_k}{k^{{1}/{p'}}} <{3}/{2}.
\end{equation}
\end{enumerate}
Set $I=\bigcup_{n\in A}I_n$, $E_M=A\cup I$, and $l_1=\max I_{m_1}$. We equip $E_M$ with the order $\preceq$ defined as follows:
\begin{itemize}[leftmargin=*]
\item if $n_1\in A$ and $n_2 \in I_{n_1}$, then $n_1\prec n_2$,
\item if $n_1, n_2\in A$, $n_1<n_2$ and $k_1\in I_{n_1}$, then $k_1\prec n_2$, and
\item if $n_1, n_2\in A$ or $n_1, n_2\in I$, then $n_1\preceq n_2$ if and only if $n_1 \le n_2$.
\end{itemize}

Let $\mathcal{A}$ be the set of all sequences $\Delta=(\delta_n)_{n\in E_M}$ with the following properties:
\begin{enumerate}[label=(R.\arabic*), leftmargin=*, widest=3]
\item For each $n\in E_M$, $\abs{\delta_n}\le 1$.
\item For each $n\in A$,
\[
\delta_n+\sum_{k\in I_n}\delta_k c_k =0,
\]
\item \label{lpunotqglcI} For each $k\in I$, $\abs{\delta_k}\le k^{-{1}/{p'}}$.
\end{enumerate}

Now consider for each sequence $\Delta=(\delta_n)_{n\in E_M}\in\mathcal{A}$, $n\in A$ and $l\in \{n\}\cup I_n$, the seminorm $\norm{\cdot}_{\Delta,n,l}$ on $\XX_M:=\FF^{E_M}$ given by
\[
f=(a_n)_{n\in E_M}\mapsto \left(\sum_{1\le j< n}\abs{ \delta_j\, a_j+\sum_{k\in I_j}\delta_k\, a_k}^p+\abs{ \sum_{\substack{j\in \{n\}\cup I_n\\ j\le l}}\delta_j\, a_j}^p\right)^{{1}/{p}}.
\]
Define the seminorm
\[
\norm{ f}_{\triangleleft}=\sup_{\Delta\in \mathcal{A}}\sup_{n\in A}\sup_{l\in \{n\}\cup I_n} \norm{f}_{\\Delta,n,l} , \quad f\in \XX_M,
\]
and then the norm
\[
\norm{ f}_{\XX_M}=\max\enbrace{ \norm{f}_{\ell_{\infty}}, \norm{ S_I(f)}_{\YY}, \norm{ f}_{\triangleleft} }, \quad f\in \XX_M.
\]

As in the proof of Theorem~\ref{thm:Miguel}, the canonical basis $\XB_M=(\xx_{M,n})_{n\in E_M}$ of $\XX_M$, ordered by $\preceq$, is a monotone normalized Schauder basis of $(\XX_M,\norm{\cdot}_{\XX_M})$ with normalized dual basis. To estimate the constants of the basis, we pick $D\subset E_M$, $\varepsilon\in \EE^D$, and $(a_n)_{n\in D}$ with $\abs{a_n}\ge 1$ for each $n\in D$. Set $D_1=D\cap A$, $D_2:=D\cap I$ and $f=\sum_{n\in D}a_n\xx_{M,n}$.

For every $\Delta\in \mathcal{A}$, $n\in A$ and $l\in \{n \} \cup B_n$, we have
\[
\norm{\Ind_{\varepsilon, D_1}}_{\Delta,n,l}^p
=\sum_{1\le j <n}\abs{ \delta_j \,\xx_{M,j}^*(\Ind_{\varepsilon, D_1})}^p+\abs{ \delta_n \, \xx_{M,n}^*(\Ind_{\varepsilon, D_1})}^p\le\abs{D_1},
\]
and
\begin{multline*}
\norm{\Ind_{\varepsilon, D_2}}_{\Delta,n,l}^p
=\sum_{1\le j< n}\abs{\sum_{k\in I_j}\delta_k \, \xx_{M,k}^*(\Ind_{\varepsilon, D_2})}^p+\abs{ \sum_{\substack{j\in \{n\}\cup I_n\\ j\le l}}\delta_j \, \xx_{M,j}^*(\Ind_{\varepsilon, D_2})}^p\\
\le \sum_{j=1}^{n}\left(\sum_{k=1}^{\abs{I_j\cap D}}\frac{1}{k^{{1}/{p'}}}\right)^p\le p^p\abs{D_2}.
\end{multline*}
Consequently, $\norm{\Ind_{\varepsilon, D}}_{\triangleleft} \le 2^{1/p} p \abs{D}^{1/p}$. Hence, using \eqref{lemmaqglclpanotqgdem2} we obtain
\begin{equation}\label{eq:Miguel9}
\norm{\Ind_{\varepsilon, D}}_{\XX_M} \le \max\{2^{1/p} p, K\} \abs{D}^{1/p}.
\end{equation}
Reasoning as in the proof of Theorem~\ref{thm:Miguel} we get
\begin{equation}\label{eq:Miguel7}
\frac{1}{K} \abs{D_2}^{{1}/{p}}\le \norm{ \Ind_{\varepsilon, D_2}}_{\YY} =\norm{ S_I\left(\Ind_{\varepsilon, D}\right)}_{\YY}\le K\norm{ S_I(f)}_{\YY} \le K \norm{f}_{\XX_M}.
\end{equation}

Set
\[
D_{1,1}=\enbrace{ n\in D_1 \colon \sum_{k\in I_n\setminus D}\frac{c_k}{k^{{1}/{p'}}}\ge \frac{1}{2}}
\]
and define for each $n\in D_{1,1}$
\[
\delta_{n}=\frac{1}{2 \sgn(a_{n})}.
\]
By definition of $D_{1,1}$, for each $n\in D_{1,1}$ there are scalars $(\delta_j)_{j\in I_n\setminus D}$ with $\abs{\delta_j}\le j^{-\frac{1}{p'}}$ for all $j\in I_n\setminus D$ such that
\[
\delta_n+\sum_{j\in I_n\setminus D}c_j\delta_j=0.
\]
Let $\delta_j=0$ for any other $j\in E_M$. Then $\Delta=(\delta_k)_{k\in E_M}\in \mathcal{A}$, and
\begin{equation}\label{paradem5}
\norm{f}_{\XX_M}^p\ge \norm{f}^p_{\Delta, m_1, l_1} = \sum_{j\in D_{1,1}}\abs{\delta_j \, a_j }^p
=\sum_{j\in D_{1,1}}\left(\frac{\abs{a_j}}{2}\right)^p\ge \frac{\abs{D_{1,1}}}{2^p}.
\end{equation}
On the other hand, by \eqref{lemmalpunoqglc condII}, for every $n\in D_{1,2}:=D_1\setminus D_{1,1} $ we have
\begin{align*}
\sum_{k\in D\cap I_n}\frac{c_k}{k^{{1}/{p'}}}\ge \frac{1}{2},
\end{align*}
so that, as in the proof of Theorem~\ref{thm:Miguel}, we deduce that $\abs{D\cap I_n} \ge 2$. Therefore $2\abs{D_{1,2}} \le \abs{D_2}$. Combining this inequality with \eqref{eq:Miguel7} and \eqref{paradem5} we get
\begin{equation}\label{eq:Miguel8}
\abs{D}^{1/p} \le \left(\frac{3}{2} \abs{D_2}+\abs{D_{1,1}}\right)^{1/p}\le \left( \frac{3 K^{2p}}{2} +2^p\right)^{1/p} \norm{f}_{\XX_M}.
\end{equation}
In the particular case that $f=\Ind_{\varepsilon,D}$, \eqref{eq:Miguel8} gives
\begin{equation}\label{eq:Miguel10}
\abs{D}^{1/p} \le \left( \frac{3 K^{2p}}{2} +2^p\right)^{1/p} \norm{\Ind_{\varepsilon,D}}_{\XX_M}.
\end{equation}
Combining \eqref{eq:Miguel9}, \eqref{eq:Miguel10} and \eqref{eq:Miguel8} gives that $\XB_M$ is LPU with constant $C$ and that \eqref{eq:FFlpBis} holds,
where

\[
C= \max\{ 2^{1/p} p,K\} \left( \frac{3 K^{2p}}{2} +2^p\right)^{1/p}.
\]
It remains to prove the lower estimate for the QGLC constant. To do so, define
\[
g=\sum_{k\in I}c_k\xx_{M,k}.
\]
Then by \eqref{eq:FFlpBis} and our choice of $A$,
\[
3 M_0M C< \abs{A}^{{1}/{p}}\le C \norm{\Ind_A}_{\XX_M}\le C \qglc \norm{\Ind_A + g}_{\XX_M}
\]
Note that $\norm{\Ind_A+g}_{\ell_\infty}=1$,
\[
\norm{S_I(\Ind_A+g)}_\YY=\norm{g}_\YY< M_0,
\]
and, for every $\Delta=(\delta_n)_{n\in E_M}\in \mathcal{A}$, $n\in A$ and $l\in\{n\}\cup I_n$,
\[
\norm{\Ind_A+g}_{\Delta,n,l}=\abs{\sum_{\substack{j\in \{n\}\cup I_n\\j\le l}}\delta_j \, c_j }\\
\le 1+\sum_{k\in I_n}\frac{c_k}{k^{{1}/{p'}}}\le 3.
\]
It follows that $\norm{\Ind_A + g}_{\XX_M}\le \max\{3,M_0\}\le 3 M_0$. Hence, $\qglc>M$, and the lemma is proved.
\end{proof}

Finally, we prove that QGLC bases need not be LUCC.

\begin{theorem}\label{theoremQGLCnotLUCC}
There is a monotone Schauder basis $\XB$ of a Banach space $\XX$ that it is QGLC and democratic (hence super-democratic) but is not LUCC. Moreover, given $1\le p<\infty$ we can choose $\XB$ so that its fundamental function verifies $\udf(m)\approx m^{1/p}$ for $m\in\NN$.
\end{theorem}

\begin{proof}
As before, it suffices to find a positive constant $C$ such that for each $M\in \NN$ there is a set $ E_M\subset_f \NN$ and a finite-dimensional normed space $\XX_M$ with monotone normalized Schauder basis $\XB_M=(\xx_{M,n})_{n\in E_M}$ satisfying \eqref{eq:FFlpBis} whose quasi-greedy for largest coefficients constant and its lower unconditionality for constant coefficients constant satisfy $\qglc\le C$ and $\lucc> M$. Let $\YB$, $(c_m)_{m=1}^\infty$ $K$ and $M_0$ as in the proof of Theorem~\ref{thm:Miguel}. Choose $m_1\in \NN$ so that
\[
\sum_{n=1}^{m_1}c_n^p>(8M_0M)^p,
\]
and let $A$ be as in the aforementioned proof. Pick $M_2>2c_{m_1}^{-1}$ and choose integer intervals $(B_n)_{n\in A}$ so that
\begin{enumerate}[label=(\arabic*), leftmargin=*, widest=2]
\item $A<B_n<B_{n+1}$ for all $1\le n\le m_1-1$, and
\item for all $n\in A$,
\begin{equation}\label{condIII}
\sum_{k=1}^{\abs{B_n}}\frac{1}{k^{{1}/{p'}}}\le M_2 c_n \le \sum_{k=1}^{\abs{B_n}+1}\frac{1}{k^{{1}/{p'}}}.
\end{equation}
\end{enumerate}
Set $B=\bigcup_{n\in A}B_n$ and $E_M=A\cup B$. Set $l_1=\max B_{m_1}$. Let $\mathcal{A}$ be the set of sequences $\Delta=(\delta_n)_{n\in E_M}$ with the following properties:
\begin{enumerate}[label=(R.\arabic*), leftmargin=*, widest=3]
\item For each $n\in E_M$, $|\delta_n|\le 1$.
\item For each $n\in A$, \label{qglcnolpuA}
\[
M_2c_n \delta_n +\sum_{k\in B_n}\delta_k=0.
\]
\item \label{qglcnolpuB}For each $n\in A$ and each $D\subset B_n$,
\[
\sum_{j\in D}\abs{\delta_j}\le \sum_{k=1}^{\abs{D}}\frac{1}{k^{{1}/{p'}}}.
\]
\end{enumerate}
We endow $E_M$ with the order $\preceq$ defined as follows:
\begin{itemize}[ leftmargin=*]
\item If $n_1,n_2\in A$ or $n_1, n_2\in B$, then $n_1\preceq n_2$ if and only if $n_1\le n_2$.
\item If $n_1\in A$ and $n_2\in B_{n_1}$, then $n_1\prec n_2$.
\item If $n_1, n_2\in A$, $n_1<n_2$ and $k\in B_{n_1}$, then $k\prec n_2$.
\end{itemize}
Consider for each $\Delta=(\delta_n)_{n\in E_M}\in\mathcal{A}$, $n\in A$ and $l\in \{n\}\cup B_n$, the seminorm $\norm{\cdot}_{\Delta,n,l}$ on $\XX_M:=\FF^{E_M}$ given by
\[
\norm{(a_j)_{j\in E_M}}_{\Delta,n,l}=\left(\sum_{1\le j< n}\abs{ \delta_j \, a_j+\sum_{k\in B_j}\delta_k\, a_k}^p+\abs{ \sum_{\substack{j\in \{n\}\cup B_n\\ n\le j\le l}}\delta_j \, a_j}^p\right)^{{1}/{p}}.
\]

Now, define for each sequence $f\in \XX_M$
\[
\norm{f}=\max\enbrace{ \norm{f}_{\ell_{\infty}}, \norm{ S_A(f)}_{\YY}, \norm{ f}_{\triangleleft} },
\]
where
\[
\norm{ f}_{\triangleleft}=\sup_{\Delta\in \mathcal{A}}\sup_{n\in A}\sup_{l\in\{n\}\cup B_n} \norm{f}_{\Delta,n,l}.
\]
The canonical basis $\XB_M=(\xx_{M,n})_{n\in E_M}$ ordered by $\preceq$ is a monotone normalized Schauder basis of $(\XX_M,\norm{\cdot}_{\XX_M})$ with normalized dual basis. To find an upper bound for the QGLC constant of $\XB_M$ we pick $D\subset E_M$, $\varepsilon=(\varepsilon_n)_{n\in D}\in \EE^D$, and $f\in \Cu$ with $\supp(f)\cap D=\emptyset$. Let $D_1=A\cap D$ and $D_2=B\cap D$. For every $\Delta\in \mathcal{A}$, $n\in A$ and $l\in \{n\}\cup B_n$, we have
\[
\norm{\Ind_{\varepsilon,D_1}}_{\Delta,n,l}^p
\le \sum_{1\le j< n}\abs{ \delta_j} ^p+\abs{ \delta_n}^p\le\abs{D_1},
\]
and
\[
\norm{\Ind_{\varepsilon,D_2}}_{\Delta,n,l}^p
\le\sum_{j\in A} \left(\sum_{k\in B_j\cap D}\abs{\delta_j}\right)^p\le \sum_{j\in A} \left(\sum_{k=1}^{\abs{B_j\cap D}}\frac{1}{k^{{1}/{p'}}}\right)^p\le p^p\abs{D_2},\]
whence we infer that
\begin{equation}\label{eq:Miguel92}
\norm{\Ind_{\varepsilon, D}}_{\XX_M} \le \max\{2^{1/p} p, K\} \abs{D}^{1/p}.
\end{equation}

Since $\YB$ is QGLC with constant $K$, using \eqref{lemmaqglclpanotqgdem2} we obtain
\begin{multline}\label{segundo1}
\frac{1}{K}\abs{D_1}^{{1}/{p}}
\le\norm{ \Ind_{\varepsilon, D_1}}_{\YY}
= \norm{ S_A\left(\Ind_{\varepsilon, D}\right)}_{\YY}\\
\le K \norm{ S_A\left(\Ind_{\varepsilon, D}+f\right)}_{\YY}
\le K \norm{\Ind_{\varepsilon, D}+f}_{\XX_M}.
\end{multline}
Let $A_2=\{n\in \NN \colon D\cap B_n\not=\emptyset\}$.
For every $n\in A_2$, let $(k_j)_{j=1}^{\abs{D\cap B_n}}$ be the increasing enumeration of $D\cap B_n$. Define $(\delta_k)_{k\in D\cap B_n}$ by $\delta_{k_j}=j^{-{1}/{p'}} / \varepsilon_{k_j}$ for $j=1$, \dots $\abs{D\cap B_n}$, and
\[
\delta_n=-\frac{1}{M_2c_n}\sum_{k\in D\cap B_n}\delta_k.
\]
For any other $k\in E_M$ set $\delta_k=0$. It follows from \eqref{condIII} that $\abs{\delta_n}\le 1$ for all $n\in A_2$. Moreover, since $M_2c_n>2$, for such $n$ we have
\[
\abs{\delta_n}\le \frac{1}{2}\sum_{k\in D\cap B_n}\abs{\delta_k}\le \frac{1}{2}\sum_{k=1}^{\abs{D\cap B_n}}\frac{1}{k^{{1}/{p'}}}.
\]
By construction $\Delta=(\delta_j)_{j\in E_M}\in \mathcal{A}$. Therefore,
\begin{align*}
\norm{ \Ind_{\varepsilon, D}+f}_{\XX_M}^p &\ge \norm{\Ind_{\varepsilon, D}+f }_{\Delta,m_1,l_1}^p\\
&=\sum_{n\in A_2}\abs{\delta_n\, \varepsilon_n+\sum_{k=1}^{\abs{D\cap B_n}}\frac{1}{k^{\frac{1}{p'}}} }^p \\
&\ge \sum_{n\in A_2}\abs{-\abs{\delta_n} +\sum_{k=1}^{\abs{D\cap B_n}}\frac{1}{k^{{1}/{p'}}} }^p\\
&\ge \frac{1}{2^p}\sum_{n\in A_2}\left( \sum_{k=1}^{\abs{D\cap B_n}}\frac{1}{k^{{1}/{p'}}} \right)^p\\
&\ge \frac{\abs{D_2}}{(4p)^p},
\end{align*}
which jointly with \eqref{segundo1} yields
\begin{equation}\label{Miguel100}
\abs{D}^{1/p} \le \enpar{K^{2p} + (4p)^p}^{1/p} \norm{ \Ind_{\varepsilon, D}+f}_{\XX_M}.
\end{equation}
Combining \eqref{eq:Miguel92} and \eqref{Miguel100} gives a constant $C=C(p,K)$ such that $\XB_M$ is QGLC with constant $C$ and satisfies \eqref{eq:FFlpBis}.

In order to obtain a lower bound for its LUCC constant, set
\[
f_0=\Ind_{B}+\sum_{n\in A}M_2c_n\xx_{M,n}, \quad f_1:=\Ind_{E_M},
\]
so that $\norm{ f_1}_{\XX_M}\le\lucc \norm{ f_0}_{\XX_M}$. By the definition of $\mathcal{A}$ and using \eqref{condIII},
\begin{align*}
\left\Vert f_0\right\Vert_{\triangleleft}&=\sup_{\Delta\in\mathcal{A}}\sup_{n\in A}\sup_{l\in \{n\}\cup B_n}\abs{\sum_{\substack{j\in \{n\}\cup B_n\\ n\le j\le l}}\delta_j\xx_j^*\left(f_0\right)}\\
&\le \sup_{n\in A}\left(M_2c_n+\sum_{k=1}^{\abs{B_n}}\frac{1}{k^{{1}/{p'}}}\right)\\
&\le 2M_2c_1<2M_2.
\end{align*}
Also, $\norm{ S_A(f_0)}_{\YY}\le M_2M_0$. Therefore, $\norm{ f_0}_{\XX_M}\le 2M_2M_0$.

For $n\in A$ fixed let $(k_j)_{j=1}^{\abs{B_n}}$ be the increasing enumeration of $B_n$. For each $1\le j\le \abs{B_n}$, set $\delta_{k_{j}}=j^{-{1}/{p'}}$. Put
\[
\delta_n:=-(M_2c_n)^{-1}\sum_{j=1}^{\abs{B_n}}\frac{1}{j^{{1}/{p'}}}.
\]
Then $\Delta=(\delta_k)_{k\in \XX_M}\in \mathcal{A}$. Since $M_2c_n>2$ for $n\in A$, taking into consideration our choice of $A$ and inequality \eqref{condIII} we deduce that
\begin{align*}
\norm{ f_1}_{\XX_M}&\ge \norm{f_1}_{\Delta,m_1,l_1}\\
& = \left(\sum_{n\in A}\abs{ \delta_n +\sum_{k\in B_n}\delta_k }^p\right)^{{1}/{p}}\\
&\ge \left(\sum_{n\in A} \left(-\frac{1}{2}\sum_{j=1}^{\abs{B_n}}\frac{1}{j^{{1}/{p'}}}+\sum_{j=1}^{\abs{B_n}}\frac{1}{j^{{1}/{p'}}}\right)^p\right)^{{1}/{p}}\\
&\ge \left(\sum_{n\in A}\left\vert \frac{M_2c_n}{4}\right\vert^p\right)^{\frac{1}{p}}>2M_2M_0M.
\end{align*}
It follows that $\lucc>M$ and the proof is over.
\end{proof}

To close this section, we mention another kind of partial unconditionality that is implicit in \cite{DOSZ2009}*{Definition 4.1}. We say that a basis is \emph{nearly unconditional for lower coefficients} (NULC for short) if for each $t\in (0,1]$ there is a constant $C\in (0,\infty)$ such that $\norm{S_A(f)}\le C\norm{f}$ for all $s\in (0,\infty)$, all $f\in \XX$ with $f=S_{A(f,ts,s)}(f)$ and all $A\subset A(f,ts,s)$. Given $0<t\le 1$, we define $\varrho(t)$ as the smallest value of $C$. Clearly, if $\XB$ is nearly unconditional then it is also NULC with $\varrho(t)\le \phi(t)$. Also, if $\XB$ is LPU and $f$, $t$, $s$, and $A$ are as above, then
\[
\norm{S_A(f)}\le \lpu\norm{t^{-1}f}=t^{-1}\lpu\norm{f},
\]
so $\XB$ is NULC. Now Theorems~\ref{thm:LPUnotQGLC} and~\ref{theoremQGLCnotLUCC} allow us to distinguish NULC from either property. More precisely, we can state the following.

\begin{corollary}\label{cor: NULC}
Near unconditionality for lower coefficients is a strictly weaker property than either near unconditionality or lattice partial unconditionality.
\end{corollary}



\end{document}